\documentclass{aims}
\usepackage{amsmath, amssymb,amsthm}
  \usepackage{paralist}
  \usepackage{graphics} %% add this and next lines if pictures should be in esp format
  \usepackage{epsfig} %For pictures: screened artwork should be set up with an 85 or 100 line screen
\usepackage{graphicx}  \usepackage{epstopdf}%This is to transfer .eps figure to .pdf figure; please compile your paper using PDFLeTex or PDFTeXify.
 \usepackage[colorlinks=true]{hyperref}
\hypersetup{urlcolor=blue, citecolor=red}

\usepackage{placeins}
\usepackage{enumerate}

\def\currentvolume{X}
 \def\currentissue{X}
  \def\currentyear{20XX}
   \def\currentmonth{XX}
    \def\ppages{X--XX}
     \def\DOI{10.3934/xx.xx.xx.xx}

\newtheorem{theorem}{Theorem}[section]
\newtheorem{corollary}{Corollary}

\newtheorem{lemma}[theorem]{Lemma}
\newtheorem{proposition}{Proposition}

\theoremstyle{definition}
\newtheorem{definition}[theorem]{Definition}
\newtheorem{remark}{Remark}

\newcommand{\rr}{\mathbf{R}}

\newcommand{\nnn}{\mathbf{N}}

\newcommand{\eff}{{\mathrm{eff}}}

\newcommand{\be}{\begin{equation}}
\newcommand{\ee}{\end{equation}}
\newcommand{\ba}{\begin{array}}
\newcommand{\ea}{\end{array}}
\newcommand{\disp}{\displaystyle}
\newcommand{\ve}{\varepsilon}

\newcommand{\mop}[1]{\mathop{\mathrm{#1}}}

\newcommand{\set}[1]{\left\{#1 \right\}}
\newcommand{\ii}[1]{\int \limits_{#1}}

\title[Concentration of eigenfunctions of a locally periodic operator]{Concentration of eigenfunctions of a locally periodic elliptic operator with large potential in a perforated cylinder}
\author[Iryna Pankratova and Klas Pettersson]{}


 \email{iripan@hin.no}
 \email{klapet@hin.no}

\thanks{}

\begin{document}
\maketitle

% Enter the first author's name and address:
\centerline{\scshape Iryna Pankratova }
\medskip
{\footnotesize
% please put the address of the first author
 \centerline{Narvik University College,}
   \centerline{Postbox 385,}
   \centerline{ 8505 Narvik, Norway}
} % Do not forget to end the {\footnotesize by the sign }

\medskip

\centerline{\scshape Klas Pettersson}
\medskip
{\footnotesize
 % please put the address of the second  and third author
 \centerline{Narvik University College,}
   \centerline{Postbox 385,}
   \centerline{ 8505 Narvik, Norway}
}

\bigskip

% The name of the associate editor will be entered by an editorial staff
% "Communicated by the associate editor name" is not needed for special issue.
 \centerline{(Communicated by the associate editor name)}

%The abstract of your paper
\begin{abstract}
We consider the homogenization of an elliptic spectral problem with a large potential stated in a thin cylinder with a locally periodic perforation.
The size of the perforation gradually varies from point to point.
We impose homogeneous Neumann boundary conditions on the boundary of perforation and on the lateral boundary of the cylinder.
The presence of a large parameter $1/\ve$ in front of the potential and the dependence of the perforation on the slow variable give rise to the effect of localization of the eigenfunctions.
We show that the $j$th eigenfunction can be approximated by a scaled exponentially decaying function that is constructed in terms of the $j$th eigenfunction of a one-dimensional harmonic oscillator operator.
\end{abstract}

%The title of your section 1
\section{Introduction}
We study the spectral asymptotics for a second-order elliptic operator with locally periodic coefficients
\be
\label{eq:A^eps}
A^\ve = - \partial_i(a_{ij}(x_1, \frac{x}{\ve}) \partial_j ) + \frac{1}{\ve^\beta}\, c(x_1, \frac{x}{\ve}),
\ee
defined in a thin perforated cylindrical domain in $\rr^d$ of thickness of order $\ve$, as $\ve \to 0$. The size of the perforation gradually varies along the cylinder.
The effective characteristics of the perforated cylinder (rod), as well as the methods of attacking the problem, depend essentially on the value of $\beta$. We distinguish three cases: $\beta=0$, $0<\beta<2$, and $\beta = 2$. In this paper we focus on the second case, $0< \beta<2$, and for simplicity of presentation we take $\beta=1$. The asymptotics of eigenpairs is described also in the other two cases, $\beta=0, 2$.

The case $\beta = 0$ (as well as $\beta <0$) is classical. The standard homogenization methods apply, and the convergence result for the case of a bounded domain can be retrieved from that obtained in \cite[Ch. 6]{BLP-78}. For the sake of completeness we formulate this result adapted to a locally periodic geometry (see Remark~\ref{rm:beta=0}).

The case when $\beta=1$ and the potential $c(x/\ve)$ is periodic with zero average has been studied in \cite[Ch. 12]{BLP-78}. The operator is defined in a bounded domain in $\mathbf R^d$, and the Dirichlet boundary condition is imposed on the boundary of the domain. It has been shown that the first eigenvalue of the original spectral problem converges to the first eigenvalue of the homogenized operator. The localization effect does not appear, and the corresponding eigenfunctions converge weakly in $H^1$.

If the coefficients of the operator do not oscillate, one deals with the asymptotics of a singularly perturbed operator. The ground state of a singularly perturbed nonselfadjoint operator
\[
\mu^2 \nabla_i a^{ij}(x) \nabla_j + \mu b^i(x) \nabla_i + v(x),
\]
defined on a smooth compact Riemannian manifold has been investigated in \cite{Pi-98}, \cite{HoKu-06}. The limit behaviour of the first eigenpair has been studied, as $\mu \to 0$. In the case of a selfadjoint operator ($b^i=0$), the localization of the first eigenfunction takes place in the scale $\sqrt \mu$ in the vicinity of minimum points of the potential, and the limit behaviour is described by a harmonic oscillator operator. The location and rate of concentration of the eigenfunctions are directly determined by the lower-order terms without any interplay with the scale of oscillations, in coefficients or geometry (see the one-dimensional example in Section~\ref{sec:1D-example}).

A Dirichlet spectral problem for singularly perturbed operators with oscillating coefficients has been studied in the recent work \cite{PiRy-2012}, where the ground state asymptotics has been studied using the viscosity solutions technique.

There are several works that are closely related to the problem under consideration where the localization effect is observed.

In \cite{AlPi-2002} the operator with a large locally periodic potential has been considered (the case $\beta=2$). It has been assumed that the first cell eigenvalue attains a unique minimum in the domain and at this point shows nondegenerate quadratic behaviour. The authors prove that the $j$th original eigenfunction is asymptotically given as a product of a periodic rapidly oscillating function and a scaled exponentially decaying function, the former function is the first cell eigenfunction at the scale $\ve$, and the later one is the $j$th eigenfunction of the harmonic oscillator type operator at the scale $\sqrt{\ve}$. The localization appears due to the presence of a large factor in the potential and the fact that the operator coefficients depend on the slow variable.

In \cite{AlBS-2002} the result of \cite{AlPi-2002} has been generalized to the case of transport equation posed in a locally periodic heterogeneous domain. Under the assumption that the leading eigenvalue of an auxiliary periodic cell problem attains a unique minimum, the homogenization and localization have been proved. The effective problem appears to be a diffusion equation with quadratic potential stated in the whole space. This gives an example of non-elliptic PDE for which the localization phenomenon is observed.

Localization phenomenon for a Schr\"{o}dinger equation in a locally periodic medium has been considered in \cite{AlPa-2006}. The authors show that there exists a localized solution which is asymptotically given as the product of a Bloch wave and of the solution of an homogenized Schr\"{o}dinger equation with quadratic potential.

The Dirichlet spectral problem for the Laplacian in a thin 2D strip of slowly varying thickness
has been studied in \cite{FrSo-2009}. Here the localization has been observed
in the vicinity of the point of maximum thickness. The large
parameter is the first eigenvalue of 1D Laplacian in the cross-section.

In the mentioned works, under natural non-degeneracy conditions, the asymptotics of the eigenpairs was
described in terms of the spectrum of an appropriate harmonic
oscillator operator.

The paper \cite{ChPaPi-13} deals with a spectral problem for a second order
divergence form elliptic operator in a periodically perforated bounded domain with a homogeneous Fourier boundary condition on the boundary of perforation. The coefficient  $q(x)$ in the boundary operator is a function of slow argument that leads to localization of eigenfunctions. A properly normalized eigenfunction converges to a delta function supported at the minimum point of $q(x)$. The localization takes place in the scale $\ve^{1/4}$.
In this scale the leading term of the asymptotic expansion for the $j$th eigenfunction proved to be the $j$th eigenfunction of an auxiliary harmonic oscillator operator.

The present paper describes another example of a problem for which the localization of eigenfunctions takes place. We present a localization result for an operator with a large potential stated in a thin perforated rod. The perforation is supposed to be locally periodic, i.e. the size of the holes varies gradually from point to point. The effect observed is similar to that described in \cite{ChPaPi-13}. However, the local periodicity of the microstructure together with dimension reduction (the original domain is asymptotically thin) bring additional technical difficulties. We make use of the singular measures technique when it comes to the homogenization procedure, and pass to the limit without focusing on the estimates for the rate of convergence: such estimates can be obtained following the ideas in \cite{VishLust}, \cite{OlYoSh} (see also estimates for the rate of convergence in \cite{ChPaPi-13}). We show that the $j$th eigenfunction can be approximated by a scaled exponentially decaying function being the $j$th eigenfunction of a one-dimensional harmonic oscillator operator, and the localization takes place in the scale $\ve^{1/4}$. In contrast with \cite{ChPaPi-13}, where the limit delta functions are supported at the minimum point of $q(x)$, we see that there is an interplay between the potential function and the local periodicity of the perforation. A special local average of the potential decides the points of localization (see condition (H3)). The technique used involves two-scale convergence in variable spaces with singular measure (see \cite{Zh-2000}). The proof of the convergence relies on a version of a mean-value property for locally periodic functions (see Lemma~\ref{lm:MVT-0}).

The paper is organized as follows. In Section~\ref{sec:prob-statement} we formulate the problem and state the main result in a short form. We also describe the result when $\beta=0$ and $\beta=2$. Section~\ref{sec:proof} is devoted to the proof of the main theorem~\ref{th:main-Th-short}. The proof consists of several steps. In Section~\ref{sec:prelimin-results} we obtain estimates for eigenvalues of the original problem. Section~\ref{sec:two-scale-convergence} provides all the necessary definitions and statements about the two-scale convergence in spaces with singular measures. In Section~\ref{sec:rescaling}, based on the estimates for the eigenvalues, we rescale the original problem. A priori estimates for the eigenfunctions of the rescaled problem are obtained in Section~\ref{sec:a-priori-estim}. Section~\ref{sec:pass-to-the-limit} contains a passage to the two-scale limit for the rescaled problem. The convergence of spectra is justified in Section~\ref{sec:conv-spectra}. A comprehensive result is given by Theorem~\ref{th:main-Th-full}. Section~\ref{sec:MVT} contains an auxiliary mean-value property for oscillating functions in a thin perforated rod. Often it is interesting to have a look at one-dimensional situation, where the expected effect is observed, and at the same time one can get more explicit formulae without additional technicalities. Such an example is presented in Section~\ref{sec:1D-example}.

\section{Problem statement and main results}
\label{sec:prob-statement}

Let $I = (-1/2,1/2)$ and let $Q \subset \rr^{d-1}$ be a bounded Lipschitz domain containg the origin. The points in $\mathbf R^d$ are denoted by $x =(x_1, x')$.
For a small parameter $\ve > 0$, we denote a thin cylinder segment (rod) by
\begin{align*}
G_\ve & = I \times \ve Q.
\end{align*}
In what follows we assume that $\ve = 1/(2N +1)$, $N \in \nnn$. The domain $\Omega_\ve$ is then obtained by cutting out $1/\ve$ holes in
$G_\ve$:
\begin{align*}
\Omega_\ve & = \big\{ x \in G_\ve : F\big(x_1, \frac{x}{\ve}\big) > 0 \big\},
\end{align*}
where the function $F(x_1,y_1, y') \in C^{1,\alpha}(\overline I\times \overline I \times \overline Q)$, $\alpha > 0$,
is $1$-peridic in $y_1$,
$F \big|_{y = 0} = -1$,
$F\big|_{y_1= \pm 1/2} \ge \mathrm{constant} > 0$, and
$\nabla_y F\big|_{y\neq 0} \neq 0$.

Throughout the paper, $\square = \mathbb T^1 \times Q$ is the periodicity cell, where $\mathbb T^1$ is a one-dimensional torus. The boundary of the cell is $\partial \square = \mathbb T^1 \times \partial Q$.

The hypotheses on $\ve$ and $F$ make $\Omega_\ve$ a union of cells of diameter $\ve$ with precisely one hole in each, bounded away from the cell boundary.
The shape and position of the holes vary slowly along the segment with the parameter $x_1$.
An illustration of $\Omega_\ve$ is shown in Figure~\ref{fig:omegave}(a).

We decompose the boundary of $\Omega_\ve$ as $\partial \Omega_\ve = \Gamma_\ve^- \cup \Gamma_\ve \cup S_\ve \cup \Gamma_\ve^+$ where
\begin{align*}
\Gamma_\ve & = I \times \ve \partial Q, & %\tag{lateral boundary}\\
\Gamma_\ve^\pm & = \{ \pm\textstyle{\frac{1}{2}} \} \times \ve Q, & %\tag{ends}\\
S_\ve & = \big\{ x \in G_\ve : F(x_1,\frac{x}{\ve}) = 0 \big\}.%\tag{surface of holes}
\end{align*}

We denote by
\begin{align*}
Y(x_1) & = \set{ y \in \square : F(x_1,y) > 0 }
\end{align*}
the perforated periodicity cell. The boundary of the perforated cell consists of the lateral boundary and the boundary of the hole, i.e.
$\partial Y(x_1) = \partial \square \cup \{y: \,\, F(x_1, y)=0\}$. An illustration of $Y(x_1)$ is shown in Figure~\ref{fig:omegave}(b).

\begin{figure}[b]
\centering
\hspace{.05\textwidth}
\begin{minipage}[c]{.45\textwidth}
\centering\includegraphics[width=\textwidth]{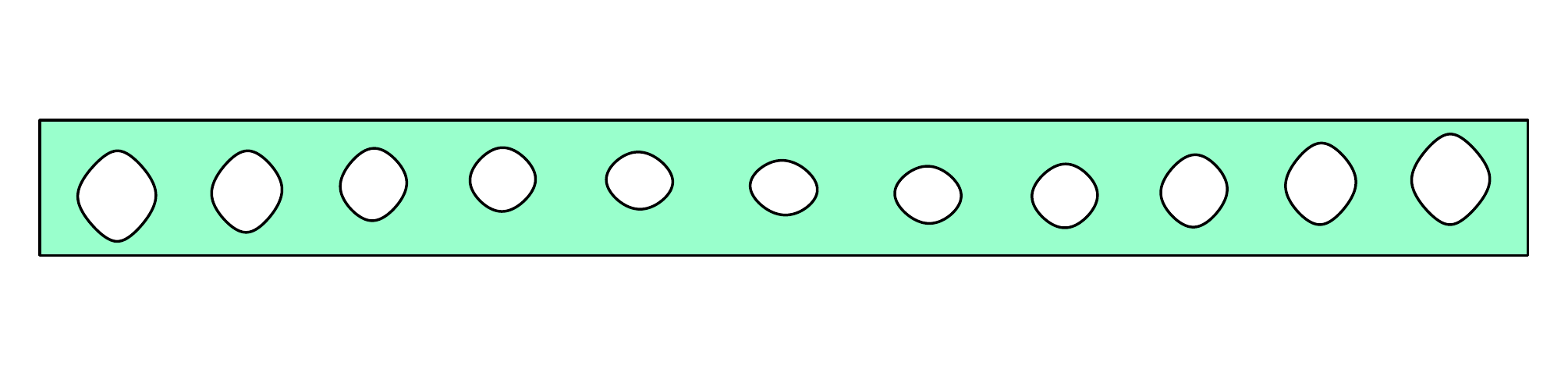}
\end{minipage}
\begin{minipage}[c]{.45\textwidth}
\centering\includegraphics[width=.55\textwidth]{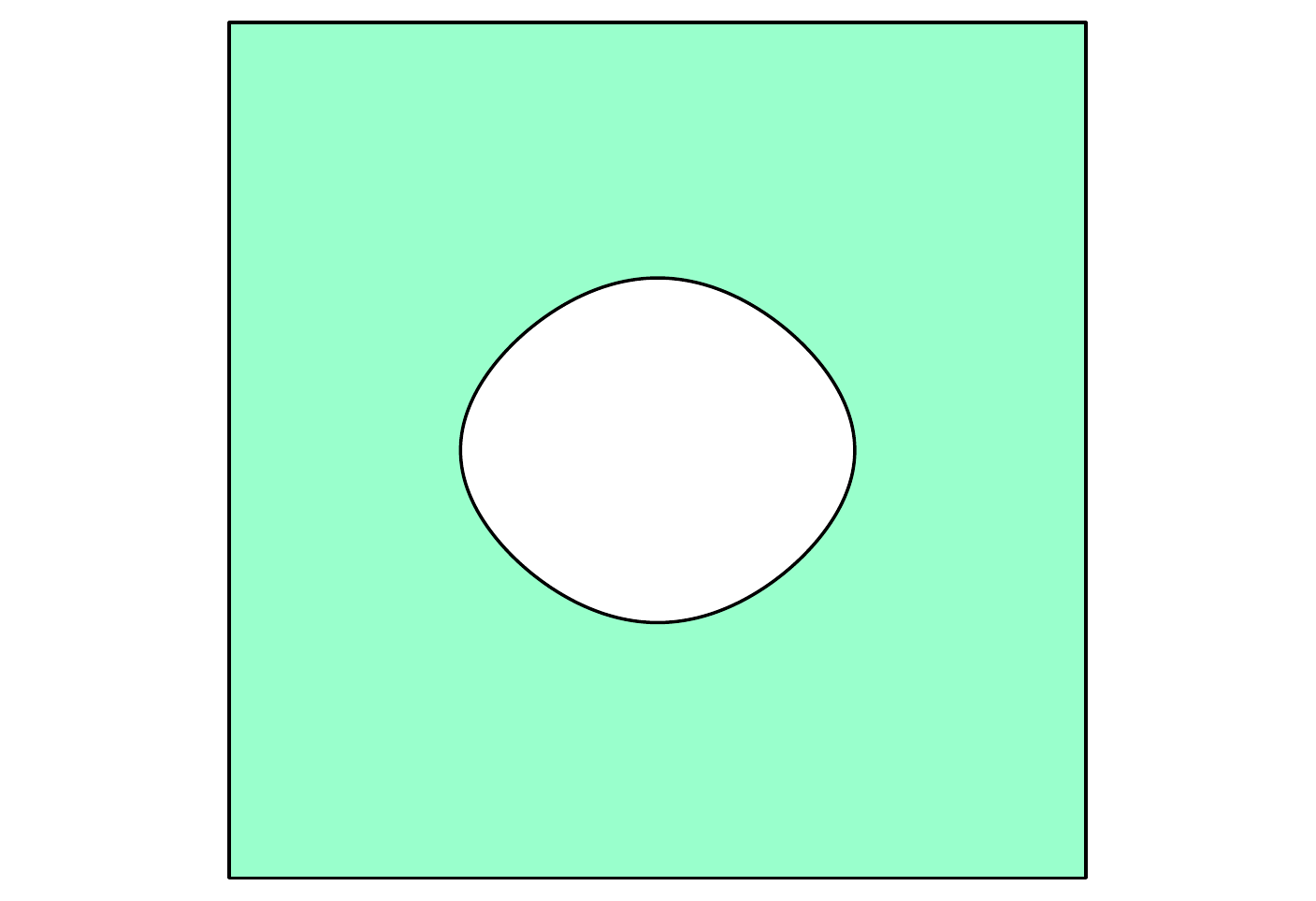}
\end{minipage}\\
\hspace{.05\textwidth}
\begin{minipage}[c]{.45\textwidth}
\centering\vspace{.25cm}(a)
\end{minipage}
\begin{minipage}[c]{.45\textwidth}
\centering\vspace{.25cm}(b)
\end{minipage}
\caption{The domains $\Omega_\ve$ (a) and $Y(x_1)$ (b).}
\label{fig:omegave}
\end{figure}

We investigate the asymptotic behaviour of solutions to the following spectral problem stated in the perforated rod $\Omega_\ve$:
\be
\label{eq:orig-prob}
\left\{
\ba{l}
\disp
- \mop{div}(a^\ve \nabla u^\ve) + \frac{1}{\ve}\, c^\ve\, u^\ve = \lambda^\ve\, u^\ve, \quad \hfill \text{ in } \Omega_\ve,
\\[3mm]
\disp
a^\ve\nabla u^\ve \cdot n_\ve = 0, \quad \hfill \text{ on } \Gamma_\ve \cup S_\ve,
\\[3mm]
u^\ve=0, \quad \hfill  \text{ on } \Gamma_\ve^\pm,
\ea
\right.
\ee
where $n_\ve$ denotes the outward unit normal.

We restrict ourselves to the following class of operators:
\begin{enumerate}[(H1)]
\item The coefficients are of the form $a^\ve(x) = a(x_1, \frac{x}{\ve})$ and $c^\ve(x) = c(x_1, \frac{x}{\ve})$,
where the functions $a_{ij}(x_1,y)\in C^{1, \alpha}(\overline{I}; C^\alpha({\square}))$, $c(x_1, y)\in C^{3}(\overline{I}; C^\alpha({\square}))$ are $1$-periodic in $y_1$, $0<\alpha<1$.
The function $c(x_1, y)$ is assumed to be positive.
\item The matrix $a(x_1, y)$ is symmetric and satisfies the uniform ellipticity condition: There is $\Lambda_0 > 0$ such that in almost everywhere in $I \times \square$,
$a(x_1, y) \xi \cdot \xi \ge \Lambda_0 |\xi|^2$ for all $\xi\in \mathbf R^d$.
\item The function
\begin{align*}
\overline{c}(x_1) := \frac{1}{|Y(x_1)|} \, \ii{Y(x_1)} c(x_1, y)\, dy
\end{align*}
has a unique global minimum at $x_1=0 \in I$ with $\bar{c}''(0)>0$.
\end{enumerate}

Denote
\begin{align*}
H_0^1(\Omega_\ve, \Gamma_\ve^\pm) = \{u \in H^1(\Omega_\ve): \,\, u=0 \,\, \mathrm{on}\,\, \Gamma_\ve^\pm\}.
\end{align*}
The weak formulation of the spectral problem \eqref{eq:orig-prob} reads:
Find $\lambda^\ve \in \mathbf{C}$ (eigenvalues) and $u^\ve \in H_0^1(\Omega_\ve, \Gamma_\ve^\pm)\setminus \{0\}$ (eigenfunctions) such that
\be
\label{eq:var-orig-prob}
\ii{\Omega_\ve} a^\ve \, \nabla u^\ve \cdot \nabla v\, dx + \frac{1}{\ve}\, \int \limits_{\Omega_\ve} c^\ve\, u^\ve\, v\, dx = \lambda^\ve \, \ii{\Omega_\ve} u^\ve\, v\, dx,
\ee
for any $v \in H_0^1(\Omega_\ve, \Gamma_\ve^\pm)$.

From the Hilbert space theory (see for example \cite{CouHi-53, Mi-65}) we have the following description of the spectrum of \eqref{eq:orig-prob}.
\begin{lemma}
Under the assumptions (H1), (H2), for any fixed $\ve >0$, the spectrum of problem \eqref{eq:orig-prob} is real and consists of a countable set of eigenvalues
\[
0 < \lambda_1^{\ve} < \lambda_2^{\ve} \leq \cdots \leq \lambda_j^{\ve} \leq \cdots \to + \infty.
\]
Each eigenvalue has finite multiplicity.
The corresponding eigenfunctions normalized by
\be
\label{eq:norm-cond-u^eps}
\ii{\Omega_\ve} u_i^{\ve}\,u_j^{\ve}\, dx = \ve^{1/4}\, \ve^{d-1} \, |Q|_{d-1} \, \delta_{ij}
\ee
form an orthogonal basis in $L^2(\Omega_\ve)$.
Here $|Q|_{d-1}$ is the  $(d-1)$-dimensional Lebesgue measure of $Q$ and $\delta_{ij}$ is the Kronecker delta.
\end{lemma}

The reason for choosing the normalization condition \eqref{eq:norm-cond-u^eps} for the eigenfunctions is to have the eigenfunctions of the rescaled problem \eqref{eq:rescaled-prob} and the limit problem \eqref{eq:eff-problem} normalized in a standard way (see \eqref{eq:norm-cond-v^eps} and \eqref{eq:norm-cond-v}, respectively).

If $c(x_1, y)$ is not positive, then the spectrum will be bounded from below by a negative constant, and all the arguments of the present paper can be repeated after shifting the spectrum by this constant. Thus, without loss of generality, we assume that $c(x_1, y)>0$.

Under the stated assumptions we study the asymptotic behaviour of the eigenpairs $(\lambda_j^\ve, u_j^\ve)$ as $\ve \to 0$.

\begin{remark}[About the extension operator]
\label{rm:extension}
For all sufficiently small $\ve$, there exists an extension operator
$P^\ve: H_0^1(\Omega_\ve, \Gamma_\ve^\pm) \to H_0^1(G_\ve, \Gamma_\ve^\pm)$
such that
\begin{align*}
\|P^\ve v\|_{L^2(G_\ve)} & \le C\, \|v\|_{L^2(\Omega_\ve)},&
\|\nabla(P^\ve v)\|_{L^2(G_\ve)} & \le C\, \|\nabla v\|_{L^2(\Omega_\ve)},
\end{align*}
where $C$ is a constant independent of $\ve$.
To construct such an operator, one can use the symmetric extensions described in \cite{LiMa-72}.

From now on we extend the solution $u^\ve$ of \eqref{eq:orig-prob} to the nonperforated rod $G_\ve$ using Remark~\ref{rm:extension} and then by zero to the infinite rod $\rr \times \ve Q$. Abusing slightly the notations, we keep the same name for the extended function.
\end{remark}

In order to formulate the main result, we introduce an effective spectral problem.
Denote by $(\nu_j, v_j(z_1))$ the $j$th eigenpair of the following one-dimensional problem:
\begin{align}\label{eq:eff-problem}
- a^\eff\, v'' + \frac{1}{2} \bar{c}''(0) \, x^2 \, v & = \nu \, v, &
v & \in L^2(\rr).
\end{align}
The coefficient $a^\eff$ is defined by
\be
\label{eq:a-eff}
a^\eff = \frac{1}{|Y(0)|} \int \limits_{Y(0)} a_{1i}(0,y)\, (\delta_{1i} + \partial_{y_i}N(y))\, dy,
\ee
where we adopt the Einstein summation convention over repeated indices.
The function $N(y) \in C^{1, \alpha}(\bar I\times \overline \Box)/\mathbf R$ is a solution of the following auxiliary cell problem:
\be
\label{eq:N}
\left\{
\begin{array}{lcr}
\displaystyle
- \mop{div}(a(0,y)\nabla N(y)) = \partial_{y_i} a_{i 1}(0, y), \quad \hfill y \in Y(0),
\\[1.5mm]
\displaystyle
a(0, y)\nabla N(y)\cdot n = -  a_{i 1}(0, y)\, n_i, \quad \hfill y \in \partial Y(0).
\end{array}
\right.
\ee
Using the symmetry and periodicity of $a(x_1, y)$, one can show that the effective coefficient $a^\eff$ is strictly positive (see the proof of Lemma~\ref{lm:conv-v^eps}). The spectral problem \eqref{eq:eff-problem} is a harmonic oscillator type problem. It is well-known that the spectrum is real and discrete:
\[
0< \nu_1 < \nu_2 < \nu_3 < \cdots < \nu_j \cdots \to +\infty.
\]
All the eigenvalues are simple, the corresponding eigenfunctions $v_k(z_1)$ can be normalized by
\be
\label{eq:norm-cond-v}
\int \limits_{\rr} v_i(z_1)\, v_j(z_1)\, dz_1 = \delta_{ij},
\ee
and form an orthonormal basis in $L^2(\mathbf R)$.

Our main result can be stated as follows.
\begin{theorem}
\label{th:main-Th-short}
Suppose that (H1)--(H3) are satisfied.
Let $(\lambda_j^\ve, u_j^\ve)$ be the $j$th eigenpair to problem \eqref{eq:orig-prob}, and $u_j^\ve$ is normalized by~\eqref{eq:norm-cond-u^eps}.
Then for any $j$,
\[
\lambda_j^\ve = \frac{\bar{c}(0)}{\ve} + \frac{\nu_j}{\sqrt{\ve}} + o(\ve^{-1/2}),
\quad \ve \to 0,
\]
where $\nu_j$ is the $j$th eigenvalue of the effective spectral problem \eqref{eq:eff-problem}, and $u_j^\ve(x)$ converges to the eigenfunction $v_j(z_1)$ corresponding to $\nu_j$ in the following sense:
\[
\lim \limits_{\ve \to 0} \frac{1}{\ve^{1/4}\, \ve^{d-1}} \,  \ii{\Omega_\ve} \big|u_j^\ve(x) - v_j\big(\frac{x_1}{\ve^{1/4}}\big)\big|^2\, dx = 0.
\]
For small enough $\ve$ all the eigenvalues $\lambda_j^\ve$ are simple.
\end{theorem}
A proof of Theorem~\ref{th:main-Th-short} is given in Section~\ref{sec:proof}.

\begin{remark}[The power of $\ve$ in the zero-order term]
Theorem~\ref{th:main-Th-short} can be generalized to the case when $1/\ve$ in the zero-order term is substituted with $1/\ve^\beta$, for $0<\beta<2$.
The localization takes place in the scale $\ve^{\beta/4}$, and the homogenized problem \eqref{eq:eff-problem} does not change. As one can observe, the cases $\beta=0, 1, 2$ are special and are naturally treated by different methods. In the remarks below we give convergence results when $\beta=0, 2$.
\end{remark}

\begin{remark}[The purely periodic case $F = F(y)$ and $c = c(y)$]
In the purely periodic case, when both perforation and the potential are periodic, the localization phenomenon is not observed ($\bar c$ is constant). This case can be treated by any classical homogenization method, for example two-scale convergence or asymptotic expansion method. The presence of perforation brings some technical issues, but does not affect the main result which is described in \cite[Ch. 12]{BLP-78}.
\end{remark}

\begin{remark}[The case of a bounded potential]
\label{rm:beta=0}
When there is no large parameter in the zero-order term in \eqref{eq:orig-prob}, the classical homogenization takes place, and the localization effect is not observed.
For the sake of completeness we present the main result in this case, the proof uses classical two-scale convergence arguments and is left to the reader.

Let $u^\ve$ be a solution of the following boundary value problem
\be
\label{eq:prob-remark-1}
\left\{
\ba{l}
\disp
- \mop{div}(a^\ve(x) \nabla u^\ve) + c^\ve(x)\, u^\ve = \lambda^\ve\, u^\ve, \quad \hfill x \in \Omega_\ve,
\\[2mm]
\disp
a^\ve(x)\nabla u^\ve \cdot n_\ve = 0, \quad \hfill x \in \Gamma_\ve \cup S_\ve,
\\[2mm]
u^\ve=0, \quad \hfill  x \in \Gamma_\ve^\pm,
\ea
\right.
\ee
where $a^\ve(x)=a(x_1, \frac{x}{\ve})$, $c^\ve(x)=c(x_1, \frac{x}{\ve})$.

We introduce an effective problem
\be
\label{eq:eff-prob-remark-1}
\left\{
\ba{l}
\disp
- (a^\eff(x_1) u')' + \bar{c}(x_1)\, u = \lambda\, u, \quad \hfill x \in I,
\\[3mm]
\disp
u\big(\pm \frac{1}{2}\big)=0.
\ea
\right.
\ee
The effective diffusion coefficient  $a^\eff$ and the potential $\bar{c}$ are given by the formulae
\begin{align*}
a^\eff(x_1) & = \frac{1}{|Y(x_1)|} \int \limits_{Y(x_1)} a_{1j}(x_1, y)(\delta_{1j} + \partial_{y_j} N_1(x_1, y))\, dy,\\
\bar{c}(x_1) & = \frac{1}{|Y(x_1)|} \int \limits_{Y(x_1)} c(x_1, y)\, dy.
\end{align*}
The auxiliary function $N_1(x_1, y)$ solves the following cell problem:
\begin{align*}
%\label{eq:chi}
\left\{
\begin{array}{lcr}
\displaystyle
- \mop{div}(a(x_1,y)\nabla N_1(x_1, y)) = \partial_{y_i} a_{i 1}(x_1, y), \quad \hfill y \in Y(x_1),
\\[2mm]
\displaystyle
a(x_1, y)\nabla N_1(x_1, y)\cdot n = - a_{i 1}(x_1, y)\, n_i, \quad \hfill y \in \partial Y(x_1).
\end{array}
\right.
\end{align*}

\begin{proposition}
Let $(\lambda_j^\ve, u_j^\ve)$ be the $j$th eigenpair of problem \eqref{eq:prob-remark-1}.
Under the assumptions (H1), (H2) the following convergence result holds:
\begin{enumerate}[(i)]
\item $\lambda_j^\ve$ converges to $\lambda_j$, as $\ve \to 0$,
\item
$
\disp
\lim \limits_{\ve \to 0}\frac{1}{\ve^{d-1}}\, \int \limits_{\Omega_\ve} |u_j^\ve(x)-u_j(x_1)|^2\, dx = 0,
$
where $(\lambda_j, u_j)$ is the $j$th eigenpair of the effective spectral problem~\eqref{eq:eff-prob-remark-1}.
\end{enumerate}
\end{proposition}
\end{remark}

\begin{remark}[The case when the potential is of order $\ve^{-2}$]
\label{rm:beta=2}
A different effect appears when the zero-order term in~\eqref{eq:orig-prob} is of order $\ve^{-2}$.
This case has been considered in~\cite{AlPi-2002} for a bounded domain $\Omega$ (without perforation).
For the case of a thin domain with a locally periodic perforation the proof is to be adjusted, but the main result and the method of proof remains the same.
We formulate the convergence result in this case omitting the proof.

We study the asymptotic behaviour of the solutions $(\lambda^\ve, u^\ve)$ of the following spectral problem:
\begin{equation}
\label{eq:scaling-2}
\left\{
\begin{array}{lcr}
\displaystyle
- \mop{div}(a^\ve \nabla u^\ve) +\frac{1}{\ve^2}\, c^\ve\, u^\ve = \lambda^\ve u^\ve, \quad \hfill x \in \Omega_\ve,
\\[2mm]
\displaystyle
a^\ve \nabla u^\ve\cdot n = 0, \quad \hfill x \in \Gamma_\ve\cup S_\ve,
\\[2mm]
u^\ve=0, \quad \hfill x \in \Gamma_\ve^\pm.
\end{array}
\right.
\end{equation}
The auxiliary cell eigenproblem, now with parameter $x_1$, becomes
\begin{equation}
\label{eq-p}
\left\{
\begin{array}{l}
-\mop{div}_y(a(x_1, y)\nabla_y p) + c(x_1, y)\, = \lambda(x_1)\, p, \quad \hfill y \in Y(x_1),
\\[2mm]
a(x_1, y)\nabla_y p\cdot n = 0, \quad \hfill y \in \partial Y(x_1).
\end{array}
\right.
\end{equation}
The spectrum of the last problem is discrete, the first eigenvalue $\lambda_1(x_1)$ is simple for all $x_1$, and the corresponding eigenfunction $p_1(x_1, y)$ is H\"{o}lder continuous and can be chosen positive.
We add an assumption that determines the location and the scale of concentration:
\begin{itemize}
\item[(H4)]
The first eigenvalue of the cell problem \eqref{eq-p} has a unique minimum point $\xi \in I$. Without loss of generality, we assume that $\xi=0$.
Moreover, we assume that in the vicinity of $x_1=0$,
\[
\lambda_1(x_1)= \lambda_1(0) + \frac{1}{2}\lambda_1''(0)\, |x_1|^2 + o(|x_1|^2), \quad \lambda_1''(0)>0.
\]
\end{itemize}

Now we formulate the homogenization result in this case.
\begin{proposition}
%\label{Theorem-scaling-2}
Suppose that (H1), (H2), (H4) are fulfilled.
Let $(\lambda_j^\ve, u_j^\ve)$ be the $j$th eigenpair of problem~\eqref{eq:scaling-2},
with eigenfunctions normalized by
\[
\|u_j^\ve\|_{L^2(\Omega_\ve)}^2= \ve^{d-1}\,\ve^{1/2}\, |Q|.
\]
Denote by $(\lambda_1(x_1), p_1(x_1, y))$ the principal eigenpair of the cell eigenproblem~\eqref{eq-p}.
Then
\[
\lambda_j^\ve = \frac{\lambda_1(0)}{\ve^2} + \frac{\nu_j}{\ve} + o(\ve^{-1}), \quad \ve \to 0,
\]
and
the corresponding eigenfunctions $u_j^\ve$ are approximated by $p_1\big(0, \frac{x}{\ve}\big)\, v_j \big(\frac{x}{\sqrt{\ve}}\big)$, that is
\[
\lim \limits_{\ve \to 0} \frac{1}{\ve^{1/2}\ve^{d-1}}\int\limits_{\Omega_\ve} \big|u_j^\ve(x) - p_1\big(0, \frac{x}{\ve}\big)\, v_j \big(\frac{x}{\sqrt{\ve}}\big)\big|^2\, dx = 0,
\]
where $(\nu_j, v_j(z_1))$ is the $j$th eigenpair, under suitable normalization, of the effective one-dimensional spectral problem
\begin{align*}
%\label{eff-prob-scaling-2}
-a^\eff v'' + (c^\eff + \frac{1}{2}\lambda_1''(0)|z_1|^2)\, v = \nu\, v, \quad
v \in L^2(\mathbf{R}),
\end{align*}
with
\[
c^\eff = - \frac{1}{|Y(0)|} \int \limits_{Y(0)} p_1(0,y)(\partial_{x_1}a_{1j}\partial_{y_j}p_1 + a_{1j}\partial_{x_1}\partial_{y_j}p_1 + \partial_{y_i}(a_{i1}\partial_{x_1}p_1)(0,y) dy,
\]
and $a^\eff$ is the strictly positive constant defined by
\[
a^\eff = \frac{1}{|Y(0)|} \int \limits_{Y(0)} p_1(0, y)^2\, a_{1j}(0, y) (\delta_{1j} + \partial_j N_2(y))\, dy,
\]
with the functions $N_2$ solving the auxiliary cell problem
\[
\left\{
\begin{array}{lcr}
\displaystyle
- \mop{div}(p_1(0, y)^2 a(0,y)\nabla N_2) = \partial_{y_i} (p_1(0, y)^2 a_{i 1}(0, y)), \quad \hfill y \in Y(0),
\\[2mm]
\displaystyle
p_1(0, y)^2 a(0, y)\nabla N_2\cdot n = - p_1(0, y)^2 a_{i 1}(0, y)\, n_i, \quad \hfill y \in \partial Y(0).
\end{array}
\right.
\]
\end{proposition}
\end{remark}

\begin{remark}[The flatness property in hypothesis (H3)]
In (H3) we assume that $\bar c''(0)$ is the first nonvanishing derivative of $\bar c(x_1)$ at the minimum point.
If instead the flatness of $\bar c(x_1)$ at the minimum point is $k \ge 2$, that is $\bar c^{(k)}(0) > 0$ is the first nonvanishing derivative, then $k$ is necessarily even and the rate of concentration will be $\ve^{ -1/(k + 2) }$.
This is apparent in the proof of Lemma~\ref{lm:MVT-0}.
We see that the flatter the averaged potential is, the slower the rate concentration of the eigenfunctions is.
The effective problem in this case reads
\begin{align}
\label{eq:eff-prob-remark-2}
- a^\eff\, v'' + \frac{1}{k!} \bar{c}^{(k)}(0) \, x^k \, v & = \nu \, v, &
v & \in L^2(\rr).
\end{align}
The effective coefficient $a^\eff$ is defined by \eqref{eq:a-eff}. Due to the growing potential, the operator is coercive, the spectrum of \eqref{eq:eff-prob-remark-2} is real and discrete. All the eigenvalues are positive and simple.

The following result holds.
\begin{proposition}
Suppose that $\bar c$ has a unique minimum point at $x_1=0$, and $\bar c^{(k)}(0) > 0$ is the first nonvanishing derivative.
Let $(\lambda_j^\ve, u_j^\ve)$ be the $j$th eigenpair to problem \eqref{eq:orig-prob}, and $u_j^\ve$ is normalized by $\|u_j^\ve\|_{L^2(\Omega_\ve)}^2=\ve^{1/(k+2)}\, \ve^{d-1} \, |Q|_{d-1}$.
Then for any $j$,
\[
\lambda_j^\ve = \frac{\bar{c}(0)}{\ve} + \frac{\nu_j}{\ve^{2/(k+2)}} + o(\ve^{- 2/(k+2)}),
\quad \ve \to 0,
\]
where $\nu_j$ is the $j$th eigenvalue of the effective spectral problem \eqref{eq:eff-prob-remark-2}.

The corresponding eigenfunction $u_j^\ve(x)$ converges to the eigenfunction $v_j(z_1)$ corresponding to $\nu_j$ in the following sense:
\[
\lim \limits_{\ve \to 0} \frac{1}{\ve^{1/(k+2)}\, \ve^{d-1}} \,  \ii{\Omega_\ve} \big|u_j^\ve(x) - v_j\big(\frac{x_1}{\ve^{1/(k+2)}}\big)\big|^2\, dx = 0.
\]
\end{proposition}
\end{remark}

\begin{remark}[The location of the minimum point in hypothesis (H3)]
Assuming that $\bar c(x_1)$ has its unique minimum at $x_1 = 0 \in I$
means that we treat the general case of minimum point in the interior of $I$.
On the other hand, if the minimum point is attained on the boundary, at $x_1 = -\frac{1}{2}$ or $x_1 = \frac{1}{2}$, then the homogenized equation is posed on a half-space and inherits the homogeneous Dirichlet condition.
That is, equation \eqref{eq:eff-problem} should be replaced by either of
\begin{align*}
-a^\mathrm{eff} v'' + \textstyle{\frac{1}{2}}\bar c''\big(\pm \textstyle{\frac{1}{2}}\big)\,x^2 v & = \nu v, & v \in H^1_0\big(\big\{ x \in \rr : x \gtrless \pm \textstyle{\frac{1}{2}} \big\}\big).
\end{align*}
\end{remark}

%%%%%%%%%%%%%%%%%%%%%%%%%%%%%%%%%%%%%%%%%%%%%%%%%%%%%%%%%%%%%%%%%%%%%%%%%%%%%%%%%%%%%%%%%%%%%%%%%%%%%%%%%%%%%%%%%%%%%%%%%%%%%%%%%%%%%%%%%%%%%%%%%%%
\section{Proof of Theorem~\ref{th:main-Th-short}}
\label{sec:proof}

In this section we prove Theorem~\ref{th:main-Th-short}. The proof is organized as follows.
First we derive estimates for the eigenvalues $\lambda_j^\ve$ (Section~\ref{sec:prelimin-results}).
Based on these estimates, we make a suitable change of variables and rescale the original problem (Section~\ref{sec:rescaling}).
In Section~\ref{sec:a-priori-estim} we obtain a priori estimates for the rescaled spectral problem and then pass to the limit in Section~\ref{sec:pass-to-the-limit}.
Lastly, we deduce the convergence of spectra (Section~\ref{sec:conv-spectra}).

\subsection{A priori estimates for eigenvalues}%$\lambda_n^\ve$}
\label{sec:prelimin-results}

The goal of this section is to obtain estimates for the eigenvalues $\lambda_j^\ve$ of problem \eqref{eq:orig-prob}. The following result provides not only the information about the behaviour of the eigenvalues, but also gives an idea about the right scaling for eigenfunctions.

\begin{lemma}
\label{lm:est-lambda^eps}
Suppose that (H1)--(H3) are satisfied.
Then there exist positive constants $C_1$ and $C_2(j)$ that are independent of $\ve$ such that
\be
\label{eq:lambdaestimate}
-\frac{C_1}{\sqrt \ve} < \lambda_j^\ve - \frac{\bar c(0)}{\ve} \le \frac{C_2(j)}{\sqrt \ve},
\ee
for all $j$ and all sufficiently small $\ve$.
\end{lemma}
\begin{proof}

We begin by estimating the first eigenvalue $\lambda_1^\ve$.
By the minmax principle (see \cite{CouHi-53}, \cite{Mi-65}),
\be
\label{eq:min-max-lambda_1}
 \lambda_1^\ve = \inf \limits_{v \in H_0^1(\Omega_\ve, \Gamma_\ve^\pm)\setminus \{0\}}
 \frac{\int\limits_{\Omega_\ve} a^\ve \nabla v \cdot \nabla v \,dx + \frac{1}{\ve} \int\limits_{\Omega_\ve} c^\ve v^2 \,dx }{\|v\|_{L^2(\Omega_\ve)}^2}.
\ee

To obtain a rough estimate from above one can take a test function $v(x_1) \in C_0^\infty(I)$ in \eqref{eq:min-max-lambda_1} and get
\[
\lambda_1^\ve \le \frac{C}{\ve}
\]
with some constant $C$ independent of $\ve$.

In order to obtain the claimed estimate, we need to make a better choice of test function.
Applying Lemma~\ref{lm:MVT-0} in \eqref{eq:min-max-lambda_1} gives
\begin{align}
\label{eq:aux-1}
\lambda_1^\ve  \le \|w_\ve\|_{L^2(\Omega_\ve)}^{-2} \Big(\int \limits_{\Omega_\ve} a^\ve \nabla w_\ve \cdot \nabla w_\ve \,dx & + \frac{1}{\ve |\Box|} \int \limits_{G_\ve} |Y(x_1)|\, \bar c(x_1) w_\ve^2 \,dx \\
 & + C\, \|w_\ve\|_{L^2(G_\ve)}\, \|\nabla w_\ve\|_{L^2(G_\ve)}\Big),\notag
\end{align}
for any $w_\ve \in H_0^1(\Omega_\ve, \Gamma_\ve^\pm)$.

One can see that to minimize the expression on the right hand side, the function $w_\ve$ should concentrate in the vicinity of the minimum point of $\bar c$. We choose $w_\ve = v(x_1/\ve^\gamma)$, $v \in C_0^\infty(I)$, $\|v\|_{L^2(\mathbf R)}=1$, with some $0<\gamma<1$.
With the help of the assumptions (H2)--(H3) we obtain:
\begin{align*}
\int\limits_{\Omega_\ve} a^\ve \nabla v(\frac{x_1}{\ve^\gamma}) \cdot \nabla v(\frac{x_1}
       {\ve^\gamma}) \,dx & = O(\ve^{d-1-\gamma}),\\
\int\limits_{G_\ve} |Y(x_1)|\, (\bar c(x_1) - \bar c(0))\, v(\frac{x_1}{\ve^\gamma})^2 \,dx & =  O(\ve^{d-1+3\gamma}), \\
\int\limits_{\Omega_\ve} v(\frac{x_1}{\ve^\gamma})^2 \,dx  = \frac{1}{|\Box|}\int \limits_{G_\ve} |Y(x_1)|v(\frac{x_1}{\ve^\gamma})^2 \,dx + O(\ve^d)& = O(\ve^{d-1+\gamma}).
\end{align*}
Using the above estimates in \eqref{eq:aux-1} gives
\begin{align*}
%\label{eq:aux-2}
\lambda_1^\ve \le \frac{\bar c(0)}{\ve} + C( \ve^{-2\gamma}  + \ve^{2\gamma - 1} ).
\end{align*}
The best choice of $\gamma$ for the considered type of test function is therefore $\gamma=1/4$, and
\[
\lambda_1^\ve \le \frac{\bar c(0)}{\ve} + C\,\ve^{-1/2}.
\]
To be able to estimate the following eigenvalues $\lambda_j^\ve$, $j=2, 3, \cdots$, one should choose a test function that concentrates in a vicinity of $x_1=0$ and is orthogonal to the first $j-1$ eigenfunctions $u_k^\ve$, $k=1, \cdots, j-1$.
In the case of the second eigenvalue, for example, it will be
\[
w_\ve(x) = v(\frac{x_1}{\ve^\gamma}) - u_1^\ve(x)\, \int \limits_{\Omega_\ve} v(\frac{x_1}{\ve^\gamma}) \, u_1^\ve(x)\, dx,
\]
which for a suitable $v$ is not zero.
The upper bound for $\lambda_j^\ve$ then follows by similar arguments as $\lambda_1^\ve$.
We omit the details.

We proceed with the estimate from below in~\eqref{eq:lambdaestimate} for $\lambda_1^\ve$.
Let $u_1^\ve$ be a first eigenfunction normalized by \mbox{$\| u_1^\ve \|_{L^2(G_\ve)} = 1$}.
Then by the positive definiteness of $a(x_1, y)$,
\begin{align}\label{eq:aux5}
\lambda_1^\ve & > \frac{1}{\ve} \frac{\int\limits_{\Omega_\ve} c^\ve (u_1^\ve)^2 \,dx}{\int\limits_{\Omega_\ve} (u_1^\ve)^2 \,dx}.
\end{align}
By Lemma~\ref{lm:MVT-0}, using that $\bar c(x_1)$ has its unique minimum at $0$,
\begin{align*}
\int\limits_{\Omega_\ve} c^\ve (u_1^\ve)^2 \, dx & > \frac{\bar c(0)}{|\Box|}\int\limits_{G_\ve} |Y(x_1)| (u_1^\ve)^2 \, dx  -C\ve \| \nabla u_1^\ve \|_{L^2(G_\ve)},\\
\int\limits_{\Omega_\ve} (u_1^\ve)^2 \, dx  & = \frac{1}{|\Box|}\int\limits_{G_\ve} |Y(x_1)| (u_1^\ve)^2 \, dx + O(\ve   \| \nabla u_1^\ve \|_{L^2(G_\ve)}).
\end{align*}
From the upper bound in~\eqref{eq:lambdaestimate},
\begin{align*}
\| \nabla u_1^\ve \|_{L^2(G_\ve)}^2 & = O(\lambda_1^\ve) \subset O(\ve^{-1}).
\end{align*}
It follows from \eqref{eq:aux5} that
\begin{align*}
\lambda_j^\ve \ge \lambda_1^\ve \ge \frac{\bar c(0)}{\ve} - C \ve^{- 1/2},
\end{align*}
which completes the proof.
\end{proof}

\begin{remark}[{Concentration of eigenfunctions}]
In the derivation of the upper bound for $\lambda_j^\ve$, we used a test function concentrated at $x_1=0$, namely a test function of the form $v(\ve^{-1/4}x_1)$.
Using the obtained estimates we can immediately deduce that the eigenfunctions of problem \eqref{eq:orig-prob} do concentrate in the vicinity of the minimum point of $\bar{c}(x_1)$.
This is an independent observation which will not be used in the proof of Theorem~\ref{th:main-Th-short}.
\begin{proposition}
\label{lm:concentration}
The eigenfunctions $u_j^\ve$ normalized by \mbox{$\| u_j^\ve \|_{L^2(\Omega_\ve)} = 1$} concentrate in the vicinity of the unique minimum point of $\bar c$ in the sense that for all $\delta > 0$ there exists $\ve_0(j) > 0$ such that
for all positive $\ve < \ve_0$, \mbox{$\| u^\ve_j \|_{L^2(\Omega_\ve \setminus \{ x : |x_1| < \delta \})} < \delta$}.
\end{proposition}
\begin{proof}
Suppose that $\| u_j^\ve \|_{L^2(\Omega_\ve)} = 1$ and $u^\ve_j$ does not concentrate.
Then there exists $\delta > 0$ such that for all $\ve_0 > 0$ there exists a positive
$\ve < \ve_0$ such that
\begin{align}\label{eq:supnot}
\int\limits_{G_\ve \setminus \{ x : |x_1| < \delta \}} |Y(x_1)| (u^\ve_j)^2  \, dx > \delta.
\end{align}
Let $\ve_0$ be small enough for Lemma~\ref{lm:est-lambda^eps} (and therefore also Lemma~\ref{lm:MVT-0}) to apply with $\ve < \ve_0$.
By normalization, $\| \nabla u_j^\ve \|_{L^2(G_\ve)} \le C\ve^{-1/2}$, and
Lemma~\ref{lm:MVT-0} gives
\begin{align*}
\lambda_j^\ve & > \frac{1}{\ve} \frac{ \int\limits_{\Omega_\ve} c^\ve (u_j^\ve)^2 \,dx }{ \int\limits_{\Omega_\ve} (u_j^\ve)^2 \,dx }
\ge \frac{1}{\ve} \frac{\int\limits_{G_\ve}  |Y(x_1)| \bar c(x_1) (u_j^\ve)^2 \,dx}{{ \int\limits_{G_\ve} |Y(x_1)|(u_j^\ve)^2 \,dx }} - C\ve^{ - \frac{1}{2}}.
\end{align*}
On the other hand,
\begin{align*}
\int\limits_{G_\ve} \bar c(x_1) |Y(x_1)| (u_j^\ve)^2 \,dx
& >
\bar c(0)  \int\limits_{G_\ve} |Y(x_1)|(u_j^\ve)^2 \,dx
+
\xi_\delta \int \limits_{G_\ve \setminus \{ x : |x_1| < \delta \}} |Y(x_1)| (u^\ve_j)^2  \, dx,
\end{align*}
where $\xi_\delta := \inf_{x_1 \in I \setminus (-\delta,\delta)} (\bar c(x_1) - \bar c(0)) > 0$ gives under hypothesis \eqref{eq:supnot},
\begin{align}\label{eq:absurdity}
\lambda_j^\ve - \frac{\bar c(0)}{\ve} > \frac{\delta \xi_\delta}{\ve} - C\ve^{- \frac{1}{2}}.
\end{align}
Since $\delta \xi_\delta > 0$ the inequality \eqref{eq:absurdity} contradicts the estimate from above in Lemma \ref{lm:est-lambda^eps} for any choice of $\ve_0$ small enough.
Proposition~\ref{lm:concentration} is proved.
\end{proof}
\end{remark}

%%%%%%%%%%%%%%%%%%%%%%%%%%%%%%%%%%%%%%%%%%%%%%%%%%%%%%%%%%%%%%%%%%%%%%%%%%%%%%%%%%%%%%%%%%%%%%%%%%%%%%%%%%%%%%%%%%%%%%%%%%%%%%%
%%%%%%%%%%%%%%%%%%%%%%%%%%%%%%%%%%%%%%%%%%%%%%%%%%%%%%%%%%%%%%%%%%%%%%%%%%%%%%%%%%%%%%%%%%%%%%%%%%%%%%%%%%%%%%%%%%%%%%%%%%%%%%%

\subsection{Singular measures and two-scale convergence}
\label{sec:two-scale-convergence}
Since the domain under consideration is asymptotically thin, it is convenient to use the singular measures technique, which was introduced independently by V.~Zhikov in \cite{Zh-2000} (analytical approach) and by G.~Bouchitt\'{e}, I.~Fragal\`{a} in \cite{BouFra-01} (geometrical approach).
We will follow the approach presented in \cite{Zh-2000} (see also \cite{ChPiSh-07}).
For the reader's convenience we include the essential definitions and main results from the theory of spaces with singular measures adapted to our case.
All the proofs follow the lines of the corresponding results in \cite{Zh-2000} and \cite{ChPiSh-07}, and are not reproduced here.

We define a Radon measure on $\rr^d$ by
\begin{equation}
\label{def:measure}
\mu_\ve(B) = \frac{\ve^{-\frac{3}{4}(d-1)}}{|Q|_{d-1}} \int \limits_{B} \chi_{\ve^{-\frac{1}{4}}G_\ve}(x) \, dx,
\end{equation}
for all Borel sets $B$,
where $\chi_{\ve^{-\frac{1}{4}}G_\ve}(x)$ is the characteristic function of the rescaled nonperforated rod $\ve^{-\frac{1}{4}}G_\ve$; $dx$ is a usual $d$-dimensional Lebesgue measure. Then $\mu_\ve$ converges weakly to the measure $\mu_\ast= dx_1 \times \delta(x')$, as $\ve \to 0$.
Indeed, let $\varphi \in C_0(\rr^d)$ and let $\ve$ be small enough such that the projection of $\mop{supp}\varphi$ on $\mathbf R$ is a subset of $\ve^{-1/4}I$.
Then
\begin{align*}
& \left| \, \int\limits_{\rr^d} \varphi(x) \, dx_1 \times d\delta(x') - \int\limits_{\rr^d} \varphi(x) \, d\mu_\ve(x) \right| =
\left| \int\limits_{\rr}  \varphi(x_1,0) \chi_{\ve^{-1/4}I}(x_1) \, dx_1 \right.
\\
& \quad
\left. - \ve^{-\frac{3(d-1)}{4}}\frac{1}{|\Box|}\int\limits_{\rr}\int\limits_{\rr^{d-1}} \varphi(x) \chi_{\ve^{-1/4}I}(x_1) \chi_{\ve^{3/4}Q}(x') \, dx_1 dx' \right| \\
& \quad =
\ve^{-\frac{3(d-1)}{4}} \frac{1}{|\Box|}\left| \int\limits_{\rr}\int\limits_{\rr^{d-1}} \big(\varphi(x_1,0) - \varphi(x)\big) \chi_{\ve^{-1/4}I}(x_1) \chi_{\ve^{3/4}Q}(x') \, dx_1dx' \right|.
\end{align*}
Let $\gamma > 0$ be given and let $\ve$ be small enough such that $x' \in \ve^{3/4}Q$ implies $|\varphi(x_1,0) - \varphi(x)| < \gamma$ using the uniform continuity of $\varphi$.
Then
\begin{align*}
\left| \, \int\limits_{\rr^d} \varphi(x) \, dx_1 \times d\delta(x') - \int\limits_{\rr^d} \varphi(x) \, d\mu_\ve(x) \right|
& \le \gamma|I|.
\end{align*}
Since $\gamma$ was arbitrary, we conclude that $d\mu_\ve$ converges weakly to $dx_1 \times \delta(x')$.

For any $\ve$, the space of Borel measurable functions $g(x)$ such that
\[
\int \limits_{\mathbf{R}^d} g(x)^2\, d\mu_\ve(x) < \infty,
\]
is denoted by $L^2(\mathbf{R}^d, \mu_\ve)$.

Let us also recall the definition of the Sobolev space with measure.
\begin{definition}
A function $g\in L^2(\mathbf{R}^d, \mu_\ve)$ is said to belong to the space $H^1(\mathbf{R}^d, \mu_\ve)$ if there exists a vector function $z \in L^2(\mathbf{R}^d, \mu_\ve)^d$ and a sequence $\varphi_k\in C_0^\infty(\mathbf{R}^d)$ such that
\[
\varphi_k \to g \quad \mbox{in} \,\, L^2(\mathbf{R}^d,\mu_\ve), \quad k \to \infty,
\]
\[
\nabla \varphi_k \to z \quad \mbox{in} \,\, L^2(\mathbf{R}^d,\mu_\ve)^d, \quad k \to \infty.
\]
In this case $z$ is called a gradient of $g$ and is denoted by $\nabla^{\mu_\ve} g$.
\end{definition}
Since in our case the measure $\mu_\ve$ is a weighted Lebesgue measure, we have $\nabla^{\mu_\ve}g =\nabla g$ and
the space $H^1(\mathbf{R}^d, \mu_\ve)$ is equivalent to the usual Sobolev space $H^1(\ve^{-1/4}G_\ve)$.

The spaces $L^2(\mathbf R, \mu_\ve)$ and $H^1(\mathbf R^d, \mu_\ast)$ are defined in a similar way, however the $\mu_\ast$-gradient is not unique and is defined up to a gradient of zero. In this case the subspace of vectors of the form $(0, \psi_2(z_1), \cdots, \psi_d(z_1))$, $\psi_j \in L^2(\mathbf R)$ is the subspace of gradients of zero $\Gamma_{\mu_\ast}(0)$ (see \cite[Ch. 2.10]{ChPiSh-07}). In other words, for $v \in H^1(\mathbf R^d, \mu_\ast)$, any $\mu_\ast$-gradient of $v$ has a form
\[
\nabla^{\mu_\ast} v(z) = (v'(z_1, 0), \psi_2(z_1), \cdots, \psi_d(z_1)), \quad \psi_j \in L^2(\mathbf R),
\]
where $v'(z_1, 0)$ is the derivative of the restriction of $v(z)$ to $\mathbf R$.

Convergence in variable spaces $L^2(\mathbf{R}^d, \mu_\ve)$ is defined as follows.
\begin{definition}
%\label{def-conv-var-spac}
A sequence $\{g^\ve(x)\} \subset L^2(\mathbf{R}^d, \mu_\ve)$ is said to converge weakly in $L^2(\mathbf{R}^d, \mu_\ve)$ to a function $g(x_1) \in L^2(\mathbf{R}^d, \mu_\ast)$, as $\ve \to 0$, if
\begin{enumerate}[(i)]
\item $\mu_\ve \rightharpoonup \mu_\ast$ weakly in $\rr^d$,
\item $\|g^\ve\|_{L^2(\mathbf{R}^d, \mu_\ve)} \le C$,
\item
for any $\varphi \in C_0^\infty(\mathbf{R}^d)$ the following limit relation holds:
\[
\lim \limits_{\ve \to 0} \int \limits_{\mathbf{R}^d} g^\ve(x) \, \varphi(x)\, d\mu_\ve(x) =
\int \limits_{\mathbf{R}^d} g(x_1)\, \varphi(x_1, 0)\, d\mu_\ast(x).
\]
\end{enumerate}
A sequence $\{g^\ve\}$ is said to converge strongly to $g(x_1)$ in $L^2(\mathbf{R}^d, \mu_\ve)$, as $\ve \to 0$, if it converges weakly and
\begin{align*}
\lim \limits_{\ve \to 0} \int \limits_{\mathbf{R}^d} g^\ve(x) \, \psi^\ve(x)\, d\mu_\ve(x)
=
\int \limits_{\mathbf{R}^d} g(x_1)\, \psi(x_1)\, d\mu_\ast(x),
\end{align*}
for any sequence $\{\psi^\ve(x)\}$ weakly converging to $\psi(x_1)$ in $L^2(\mathbf{R}^d, \mu_\ve)$.
\end{definition}
The property of weak compactness of a bounded sequence in a separable Hilbert space remains valid with respect to the convergence in variable spaces.

In the present context two-scale convergence is described as follows.
\begin{definition}
We say that $g^\ve\in L^2(\mathbf{R}^d, \mu_\ve)$ converges two-scale weakly, as $\ve \to 0$, in $L^2(\mathbf{R}^d, \mu_\ve)$ if there exists a function
${g}(x_1,y) \in L^2(\mathbf{R}^d \times \Box, \mu_\ast \times dy)$ such that
\begin{enumerate}[(i)]
\item $\mu_\ve \rightharpoonup \mu_\ast$ weakly in $\rr^d$,
\item $\|g^\ve\|_{L^2(\mathbf{R}^d, \mu_\ve)} \le C, \quad \ve>0$,
\item[$(ii)$]
the following limit relation holds:
\[
\lim \limits_{\ve \to 0} \int \limits_{\mathbf{R}^d} g^\ve(x) \, \varphi(x)\, \psi(\frac{x}{\ve^{3/4}}) d\mu_\ve(x) =
\frac{1}{|\Box|}\int \limits_{\mathbf{R}^d} \int \limits_{\Box} {g}(x_1,y)\, \varphi(x_1, 0)\, \psi(y) \, dy\, d\mu_\ast(x),
\]
for any $\varphi\in C_0^\infty(\mathbf{R}^d)$ and $\psi(y)\in C^\infty(\Box)$ periodic in $y_1$.
\end{enumerate}
We write $g^\ve \overset{2}{\rightharpoonup} g(x_1, y)$ if $g^\ve$ converges two-scale weakly to $g(x_1,y)$ in $L^2(\mathbf{R}^d, \mu_\ve)$.
\end{definition}
\begin{lemma}[Compactness]
\label{lm:compactness-1}
Suppose that $g^\ve$ satisfies the estimate
\[
\|g^\ve\|_{L^2(\mathbf{R}^d,\, \mu_\ve)} \le C.
\]
Then $g^\ve$, up to a subsequence, converges two-scale weakly in $L^2(\mathbf{R}^d, \mu_\ve)$ to some function ${g}(x_1,y) \in L^2(\mathbf{R}^d \times \Box, \mu_\ast \times dy)$.
\end{lemma}
\begin{definition}
A sequence $g^\ve$ is said to converge two-scale strongly to a function ${g}(x_1, y) \in L^2(\mathbf{R}^d \times \Box, \mu_\ast \times dy)$ if
\begin{enumerate}[(i)]
\item $\mu_\ve \rightharpoonup \mu_\ast$ weakly in $\rr^d$,
\item $g^\ve$ converges two-scale weakly to ${g}(x_1, y)$,
\item the following limit relation holds:
\[
\lim \limits_{\ve \to 0} \int \limits_{\mathbf{R}^d} (g^\ve(x))^2 d\mu_\ve(x) =
\frac{1}{|\Box|}\int \limits_{\mathbf{R}^d} \int \limits_{\Box} ({g}(x_1,y))^2\, dy\, d\mu_\ast(x).
\]
We write $g^\ve \overset{2}{\rightarrow} g(x_1, y)$ if $g^\ve$ converges two-scale strongly to the function $g(x_1,y)$ in $L^2(\mathbf{R}^d, \mu_\ve)$.
\end{enumerate}
\end{definition}
In addition to compactness, we will use the following result about the strong two-scale convergence of the characteristic functions.
\begin{lemma}
Let $\tilde \chi^\ve(z) = \chi(\ve^{1/4}z_1, z/\ve^{3/4})$ be the characteristic function of the rescaled perforated rod $\ve^{-1/4}\Omega_\ve$. For any $\varphi \in C_0(\rr^d)$, the function $\varphi(z)\tilde \chi^\ve(z)$
converges two-scale strongly to $\varphi(z)\chi(0, \zeta)$ in $L^2(\rr^d, \mu_\ve)$, as $\ve \to 0$, where $\chi(x_1, y)$ is the characteristic function of the perforated cell $Y(x_1)$, that is
\begin{align}
\chi(x_1, y) :=
\begin{cases}
1 & \text{if } F(x_1, y) > 0, \\
0 & \text{otherwise.}%\mbox{if } F(x_1, y) \le 0.
\end{cases}
\label{eq:chi}
\end{align}

\end{lemma}
\begin{proof}
The compact support of $\varphi$ makes the sequence bounded in $L^2(\rr^d, \mu_\ve)$.
The mean value property (Corollary~\ref{lm:MVT-1}) gives
\begin{align*}
& \lim_{\ve \to 0} \int\limits_{\rr^d} \varphi(z) b(\frac{z}{\ve^{3/4}}) \tilde \chi^\ve(z) \,d\mu_\ve = \int\limits_{\rr^d} \frac{1}{|\Box|}\int\limits_\square \varphi(z) b(\zeta)  \chi(0, \zeta) \,d\zeta \, dz_1 \times d\delta (z').
\end{align*}
and
\begin{align*}
& \lim_{\ve \to 0} \int\limits_{\rr^d} \varphi^2(z) \tilde \chi^\ve(z) \,d\mu_\ve
 = \int\limits_{\rr^d} \frac{1}{|\Box|}\int\limits_\square \varphi^2(z)  \chi(0, \zeta) \,d\zeta \, dz_1 \times d\delta(z'),
\end{align*}
which verifies weak and strong two-scale convergence.
\end{proof}

\subsection{Rescaled problem}
\label{sec:rescaling}
The estimate obtained in Lemma~\ref{lm:est-lambda^eps} suggests that one can rescale the original problem to make the eigenvalues bounded and then pass to the limit in the weak formulation of the problem.
The rescaling we choose is based on the special scaling of a test function used in the proof of Lemma~\ref{lm:est-lambda^eps}.
Namely, we subtract $\ve^{-1}\bar{c}(0)\, u^\ve$ from both sides of the equation in \eqref{eq:orig-prob}, perform a change of variables $z = x/\ve^{1/4}$ both in the equation and boundary conditions, and multiply the resulting equation by $\ve^{1/2}$.
In this way we obtain the rescaled problem
\begin{align}
\label{eq:rescaled-prob}
\left\{
\ba{l}
- \mop{div}(\tilde{a}^\ve(z) \nabla v^\ve) + \frac{\tilde{c}^\ve(z) - \bar c(0)}{\sqrt \ve}\, v^\ve = \nu^\ve\, v^\ve(z), \quad \hfill z \in \ve^{-1/4}\Omega_\ve,
\\[3mm]
\disp
\tilde{a}^\ve(z)\nabla v^\ve \cdot n_\ve = 0, \quad \hfill z \in \ve^{-1/4}(\Gamma_\ve \cup S_\ve),
\\[3mm]
v^\ve=0, \quad \hfill  z \in \ve^{-1/4}\Gamma_\ve^\pm.
\ea
\right.
\end{align}
%\be
%\label{eq:rescaled-prob}
%\left\{
%\ba{l}
%\disp
%- \mop{div}(\tilde{a}^\ve(z) \nabla v^\ve) + \frac{(\tilde{c}^\ve(z) - \bar c(0))}{\sqrt \ve}\, v^\ve = \nu^\ve\, v^\ve(z), \quad \hfill z \in \ve^{-1/4}\Omega_\ve,
%\\[3mm]
%\disp
%\tilde{a}^\ve(z)\nabla v^\ve \cdot n_\ve = 0, \quad \hfill z \in \ve^{-1/4}(\Gamma_\ve \cup S_\ve),
%\\[3mm]
%v^\ve=0, \quad \hfill  z \in \ve^{-1/4}\Gamma_\ve^\pm.
%\ea
%\right.
%\ee
We write
\begin{align}
\tilde{a}^\ve(z) & = a\Big(\ve^{1/4}\, z_1, \frac{z}{\ve^{3/4}}\Big), &
\tilde{c}^\ve(z) & = c\Big(\ve^{1/4}\, z_1, \frac{z}{\ve^{3/4}}\Big), \label{eq:rescaling}\\
v^\ve(z) & = u^\ve(\ve^{1/4}z), &
\nu^\ve & = \sqrt{\ve} \big(\lambda^\ve - \frac{\bar{c}(0)}{\ve}\big). \label{eq:rescaling1}
\end{align}
Note that, due to \eqref{eq:norm-cond-u^eps}, the eigenfunctions of the rescaled problem are normalized by
\be
\label{eq:norm-cond-v^eps}
\ii{\mathbf{R}^d} \tilde \chi^\ve v_i^{\ve}\,v_j^{\ve}\, d\mu_\ve = \delta_{ij},
\ee
where $\tilde \chi^\ve(z)= \chi(\ve^{1/4}z_1, z/\ve^{3/4})$ is the characteristic function of the rescaled perforated rod $\ve^{-1/4}\Omega_\ve$ defined by \eqref{eq:chi}.

Moreover, with the help of Lemma~\ref{lm:est-lambda^eps}, we deduce that the eigenvalues of the rescaled problem are bounded uniformly in $\ve$:
\be
\label{eq:est-nu^eps}
\forall j: \quad 0 < m \le \nu_j^\ve \le M_j.
\ee
To be able to pass to the limit in the weak formulation of \eqref{eq:rescaled-prob} we need a priori estimates for $v^\ve$.
Because we work in a thin domain, we expect dimension reduction, that is why we use a measure which is asymptotically singular with respect to the Lebesgue measure in $\mathbf R^d$.

%%%%%%%%%%%%%%%%%%%%%%%%%%%%%%%%%%%%%%%%%%%%%%%%%%%%%%%%%%%%%%%%%%%%%%%%%%%%%%%%%%%%%%%%%%%%%%%%%%%%%%%%%%%%%%%%%%%%%%%%%%%%%%%%%%%

\subsection{A priori estimates for eigenfunctions}
\label{sec:a-priori-estim}

In this section we will derive a priori estimates for an eigenfunction $v^\ve$ of the rescaled problem \eqref{eq:rescaled-prob}.

The weak formulation of \eqref{eq:rescaled-prob} takes the form
\be
\label{eq:weak-rescaled-prob}
\ba{l}
\disp
\int \limits_{\ve^{-\frac{1}{4}}\Omega_\ve} \tilde{a}^\ve \nabla v^\ve \cdot \nabla \phi\, dz +
\frac{1}{\sqrt{\ve}} \int \limits_{\ve^{-\frac{1}{4}}\Omega_\ve} (\tilde{c}^\ve - \bar{c}(0))\, v^\ve\, \phi\, dz
=
\nu^\ve\, \int \limits_{\ve^{-\frac{1}{4}}\Omega_\ve} v^\ve\, \phi\, dz,
\ea
\ee
for any $\phi(x_1, x') \in C_0^\infty\big(\mathbf{R}; C^\infty(\ve^{\frac{3}{4}}Q)\big)$.

Recalling the definition \eqref{def:measure} of the measure $\mu_\ve$, we can rewrite \eqref{eq:weak-rescaled-prob} as
\begin{align}
\label{eq:weak-rescaled-prob-meas}
\int \limits_{\mathbf{R}^d} \tilde{\chi}^\ve \tilde{a}^\ve \nabla v^\ve \cdot \nabla \phi\, d\mu_\ve +
\frac{1}{\sqrt{\ve}} \int \limits_{\mathbf{R}^d} \tilde{\chi}^\ve (\tilde{c}^\ve - \bar{c}(0))\, v^\ve\, \phi\, d\mu_\ve
& =
\nu^\ve\, \int \limits_{\mathbf{R}^d} \tilde{\chi}^\ve v^\ve\, \phi\, d\mu_\ve,
\end{align}
for any $\phi \in H_0^1(\mathbf{R}^d)$. As before, we assume that $v^\ve$ is extended to the nonperforated rod $\ve^{-1/4}G_\ve$ and, abusing slightly the notation, we keep the same name for the extended function (see Remark~\ref{rm:extension}).

Taking $v^\ve$ as a test function in the weak formulation and using the ellipticity of the matrix $a$ give
\begin{align}
\label{eq:aux-3}
\Lambda_0 \int \limits_{\mathbf{R}^d} |\tilde{\chi}^\ve \nabla v^\ve|^2 \, d\mu_\ve +
\frac{1}{\sqrt{\ve}} \int \limits_{\mathbf{R}^d} \tilde{\chi}^\ve (\tilde{c}^\ve - \bar{c}(0))\, (v^\ve)^2\, d\mu_\ve
& \le
\nu^\ve\, \int \limits_{\mathbf{R}^d} (\tilde{\chi}^\ve v^\ve)^2\, d\mu_\ve.
\end{align}
Since the integrand in the second term on the left in \eqref{eq:aux-3} can change sign, we cannot estimate this term directly.
Corollary~\ref{lm:MVT-1} yields
\begin{align*}
& \int \limits_{\mathbf{R}^d} \tilde{\chi}^\ve (\tilde{c}^\ve - \bar{c}(0))\, (v^\ve)^2\, d\mu_\ve
- \frac{1}{|\Box|} \int \limits_{\mathbf{R}^d} |Y(\ve^{1/4})| (\bar c(\ve^{1/4} z_1) - \bar c(0))\,(v^\ve)^2\, d\mu_\ve \\
& \quad = O(\ve\cdot \ve^{-1/4}\|v^\ve\|_{L^2(\mathbf{R}^d,\, \mu_\ve)} \|\nabla v^\ve\|_{L^2(\mathbf{R}^d,\, \mu_\ve)}).
\end{align*}
By (H3),
\begin{align*}
\bar c(\ve^{1/4} z_1) - \bar c(0) = \frac{1}{2} \bar c''(0) \sqrt \ve |z_1|^2 + o(|\ve^{1/4} z_1|^2).
\end{align*}
We cannot use the Taylor expansion here because we do not know if $v^\ve$ is localized, and thus we cannot obtain an estimate for the remainder.
Instead, we will use a quadratic equivalence which is a straightforward consequence of Taylor's theorem.
\begin{lemma}
\label{lm:equivav-funct}
Let $f, \, g \in C^k(\bar I)$, for some bounded open set $I$ in $\mathbf{R}$.
Suppose that $\xi \in I$ is a unique minimum point of both $f$ and $g$, and
$f(\xi) = g(\xi) = 0$.
Assume that $l$ is such that $1 < l < k$, and
$f^{(l)}(\xi)$ and $g^{(l)}(\xi)$ are the first
nonvanishing derivatives of the functions at $\xi$.
Then there exists a positive constant $C$ such that $C f \le g \le C^{-1}f$ on $\bar I$.
\end{lemma}
%%SUPPRESSED PROOF
%\begin{proof}
%Suppose that the first claimed inequality is not true.
%Then there exists a sequence $x_m \in \bar I$ such that $f(x_m) > mg(x_m)$.
%The boundedness of $f$ implies $f(x_m) \to 0$.
%The uniqueness of the minimum of $f$ gives $x_m \to \xi$.
%Since $g$ is nonnegative and $f(\xi) = 0$ we have $x_m \neq \xi$ by the strict inequality.
%By Taylor's theorem, using $l + 1 \le k$,
%\begin{align*}
%(x_m - \xi)^l f^{(l)}(\xi) + (x_m - \xi)^{l + 1}h(x_m)
%&  >  m((x_m - \xi)^l g^{(l)}(\xi) + (x_m - \xi)^{l + 1}\tilde h(x_m)),
%\end{align*}
%for some functions $h$, $\tilde h$ such that $h(x_m) \to 0$, $\tilde h(x_m) \to 0$.
%Since $\xi$ is a minimum point of $g$ and $g^{(l)}(\xi) \neq 0$, $l$ must be even and $g^{(l)}(\xi) > 0$.
%Since $x_m \neq \xi$, we can divide by the strictly positive $(x_m - \xi)^l g^{(l)}(\xi)$ to obtain
%\begin{align*}
%\frac{f^{(l)}(\xi) + (x_m - \xi)h(x_m)}{g^{(l)}(\xi)} > m + \frac{m(x_m - \xi)\tilde h(x_m)}{g^{(l)}(\xi)}.
%\end{align*}
%Passing to the limit in the extended real line gives
%$f^{(l)}(\xi)/g^{(l)}(\xi) = +\infty$, which cannot be true since $g^{(l)}(\xi) \neq 0$ and $f^{(l)}(\xi) \in \mathbf{R}$.
%Repeating the argument for $f$ and $g$ interchanged completes the proof.
%\end{proof}
%%END OF SUPPRESSED PROOF

By using Lemma~\ref{lm:equivav-funct} we substitute $\bar c(\ve^{1/4} z_1) - \bar c(0)$ with the equivalent quadratic function $\bar c''(0) |\ve^{1/4}z_1|^2$ in \eqref{eq:aux-3} to obtain
\begin{align*}
& \|\tilde \chi^\ve \nabla v^\ve\|_{L^2(\mathbf{R}^d,\, \mu_\ve)}^2 + \|z_1 v^\ve\|_{L^2(\mathbf{R}^d,\, \mu_\ve)}^2
\\
& + \ve^{3/4} \|v^\ve\|_{L^2(\mathbf{R}^d,\, \mu_\ve)} \|\nabla v^\ve\|_{L^2(\mathbf{R}^d,\, \mu_\ve)}
\le
C\|\tilde \chi^\ve v^\ve\|_{L^2(\mathbf{R}^d)}^2.
\end{align*}
Due to the normalization condition \eqref{eq:norm-cond-v^eps}, $\|v^\ve\|_{L^2(\mathbf{R}, \mu_\ve)} \le 1$, and we get
\be
\label{eq:apriori-est}
\|\tilde \chi^\ve \nabla v^\ve\|_{L^2(\mathbf{R}^d,\, \mu_\ve)} +
\|z_1 v^\ve\|_{L^2(\mathbf{R}^d,\, \mu_\ve)} \le C,
\ee
where $C$ is independent of $\ve$.

\begin{lemma}
\label{lm:compactness-2}
Suppose that $v^\ve \in H^1(\mathbf{R}^d, \mu_\ve)$ is such that
\begin{align*}
\int \limits_{\mathbf{R}^d} |\nabla v^\ve|^2\, d\mu_\ve
+\int \limits_{\mathbf{R}^d} |z_1 v^\ve|^2\, d\mu_\ve \le C,
\end{align*}
for some $C$ independent of $\ve$.
Then $v^\ve$ converges strongly along a subsequence in $L^2(\mathbf{R}^d, \mu_\ve)$, i.e. $v^\ve$ converges weakly in $L^2(\mathbf{R}^d, \mu_\ve)$ to some $v \in L^2(\mathbf{R}^d, \mu_\ast)$ and
\begin{align*}
\lim \limits_{\ve \to 0} \int \limits_{\mathbf{R}^d} |v^\ve|^2\, d\mu_\ve = \int\limits_{\mathbf{R}^d} |v|^2\, d\mu_\ast.
\end{align*}
\end{lemma}
Lemma~\ref{lm:compactness-2} can be proved following the lines of Lemma~4.4 in~\cite{ChPaPi-13}.

\begin{lemma}
\label{lm:compactness-v^eps}
Suppose that
\begin{align*}
\|\tilde \chi^\ve \nabla v^\ve\|_{L^2(\mathbf{R}^d,\, \mu_\ve)} +
\|z_1 v^\ve\|_{L^2(\mathbf{R}^d,\, \mu_\ve)} \le C,
\end{align*}
for some $C$ which is independent of $\ve$. Then there exists $v \in H^1(\mathbf R^d, \mu_\ast)$ such that, for a subsequence,
\begin{enumerate}[(i)]
\item
$\disp \chi(\ve^{1/4}z_1, \frac{z}{\ve^{3/4}})\, v^\ve(z) \overset{2}{\rightharpoonup} \chi(0, \zeta)\, v(z_1)$,
\item
$\disp \chi(\ve^{1/4}z_1, \frac{z}{\ve^{3/4}})\, \nabla v^\ve(z) \overset{2}{\rightharpoonup} \chi(0, \zeta)\, (\nabla^{\mu_\ast} v(z_1) + \nabla_\zeta v^1(z_1, \zeta))$, where\\
$\nabla^{\mu_\ast} v(z_1)$ is one of the gradients of $v$ with respect to the measure $\mu_\ast$, and $v^1(z_1, \zeta) \in L^2(\mathbf{R}; H^1(\Box))$ is $1$-periodic in $\zeta_1$.
\item
$v^\ve(z)$ converges strongly in $L^2(\mathbf{R}^d, \mu_\ve)$ to $v(z_1)$, as $\ve \to 0$, i.e.
\begin{align*}
\lim \limits_{\ve \to 0} \|v^\ve\|_{L^2(\mathbf{R}^d,\, \mu_\ve)} = \|v\|_{L^2(\mathbf{R},\, \mu_\ast)}.
\end{align*}
\end{enumerate}
\end{lemma}
\begin{proof}
This theorem can be proved following the lines of classical proofs of compactness results for two-scale convergence (see for example \cite{Al-1992, ChPiSh-07}), and therefore we omit the details and just indicate the main ideas.

Since the extended function $v^\ve$ as well as its gradient $\nabla v^\ve$ are bounded in $L^2(\mathbf{R}^d, \mu_\ve)$, Theorem~10.3 in~\cite{ChPiSh-07} implies that, up to a subsequence, $v^\ve \overset{2}{\rightharpoonup} v(z_1, 0)$ and $\nabla v^\ve \overset{2}{\rightharpoonup} \nabla^{\mu_\ast} v(z_1, 0) + \nabla_\zeta v^1(z_1, \zeta)$, as $\ve \to 0$. Here $v(z_1, 0)$ is the restriction of the function $v(z)=v(z_1, z')$ on $\mathbf R$, $\nabla^{\mu_\ast} v(z_1, 0)$ is the restriction of the $\mu_\ast$-gradient of $v(z)$ on $\mathbf R$.

By the mean-value property (Corollary~\ref{lm:MVT-1}) we have a strong two-scale convergence in $L^2(\mathbf{R}^d, \mu_\ve)$ of the sequence of characteristic functions on each compact $K \subset \mathbf{R}^d$:
\begin{align*}
\chi(\ve^{1/4}z_1, \frac{z}{\ve^{3/4}}) \overset{2}{\rightarrow} \chi(0,\zeta), \,\, \ve \to 0.
\end{align*}
Thus, on each compact $K \subset \mathbf{R}^d$ the statements (i) and (ii) hold.

The a priori estimate~\eqref{eq:apriori-est} gives more than just boundedness in $H^1(\mathbf{R}^d, \mu_\ve)$ (which in general do not imply strong convergence in $L^2(\mathbf{R}^d, \mu_\ve)$).
The presence of a growing weight $|z_1|$ in the $L^2$-norm guarantees that
the function $v^\ve$ is localized,
and the strong convergence in $L^2(\mathbf{R}, \mu_\ve)$ takes place
by Lemma~\ref{lm:compactness-2}.

Because of the strong convergence in $L^2(\mathbf{R}^d, \mu_\ve)$ the two-scale convergence in (i)--(ii) takes place not only on compact sets, but in the whole $\mathbf{R}^d$.
\end{proof}

%%%%%%%%%%%%%%%%%%%%%%%%%%%%%%%%%%%%%%%%%%%%%%%%%%%%%%%%%%%%%%%%%%%%%%%%%%%%%%%%%%%%%%%%%%%%%%%%%%%%%%%%%%%%%%%%%%%%%%%%%%%%%%%%%%%

\subsection{Passage to the limit}
\label{sec:pass-to-the-limit}
The main result of this section is contained in the following lemma.

\begin{lemma}
\label{lm:conv-v^eps}
Let $(\nu_j^\ve, v_j^\ve)$ be the $j$th eigenpair of the rescaled spectral problem \eqref{eq:rescaled-prob}. Then, up to a subsequence,
\begin{enumerate}[(i)]
\item
$\nu_j^\ve \to \nu_{J(j)}$, as $\ve \to 0$;
\item
$v_j^\ve$ converges strongly in $L^2(\mathbf{R}^d, \mu_\ve)$ to $v_{J(j)}(z_1)$ and moreover,
\begin{align*}
\tilde \chi^\ve\, \tilde a^\ve\, \nabla v_j^\ve \overset{2}{\rightharpoonup} \chi(0, \zeta)\, \Big((a^\eff v'(z_1), 0,\cdots, 0) + \nabla N(\zeta)\, v_{J(j)}'(z_1)\Big),\quad \ve \to 0.
\end{align*}

The pair $(\nu_{J(j)}, v_{J(j)}(z_1))$ is an eigenpair of the effective spectral problem
\begin{align}\label{eq:eff-problem-bis}
- a^\eff\, v'' + \frac{1}{2} \bar{c}''(0)\, |z_1|^2 \, v & = \nu \, v, &
v & \in L^2(\rr).
\end{align}
The coefficient $a^\eff$ is given by
\begin{align*}
%\label{eq:a-eff-bis}
a^\eff = \frac{1}{|Y(0)|} \int \limits_{Y(0)} a_{1i}(0,\zeta)\, (\delta_{1i} + \partial_{\zeta_i}N_1(\zeta))\, d\zeta.
\end{align*}
The function $N_1(\zeta)$ solves the following cell problem:
\be
\label{eq:N-bis}
\left\{
\begin{array}{lcr}
\displaystyle
- \mop{div}(a(0,\zeta)\nabla N_1(\zeta)) = \partial_{\zeta_i} a_{i 1}(0, \zeta), \quad \hfill \zeta \in Y(0),
\\[1.5mm]
\displaystyle
a(0, \zeta)\nabla N_1(\zeta)\cdot n = -  a_{i 1}(0, \zeta)\, n_i, \quad \hfill \zeta \in \partial Y(0).
\end{array}
\right.
\ee
\end{enumerate}
\end{lemma}

\begin{proof}
Since the eigenvalues $\nu_j^\ve$ of the rescaled problem are bounded (see \eqref{eq:est-nu^eps}), then, up to a subsequence, $\nu_j^\ve$ converges to some $\nu_{J(j)}$, as $\ve \to 0$.
The a priori estimate~\eqref{eq:apriori-est} together with Lemma~\ref{lm:compactness-v^eps} guarantee the convergence of the corresponding eigenfunction $v_j^\ve$ and its gradient, and it remains to prove that $(\nu_{J(j)}, v_{J(j)}(z_1))$ is an eigenpair of \eqref{eq:eff-problem}.

Let us take an eigenpair $(\nu^\ve, v^\ve)$ and show that it converges to some eigenpair $(\nu, v)$ of the effective problem.

To this end we pass to the limit, as $\ve \to 0$, in the weak formulation \eqref{eq:weak-rescaled-prob-meas}. Recall that
\be
\label{eq:weak-rescaled-prob-meas-bis}
\ba{l}
\disp
\int \limits_{\mathbf{R}^d} \tilde{\chi}^\ve \tilde{a}^\ve \nabla v^\ve \cdot \nabla \phi\, d\mu_\ve +
\frac{1}{\sqrt{\ve}} \int \limits_{\mathbf{R}^d} \tilde{\chi}^\ve (\tilde{c}^\ve - \bar{c}(0))\, v^\ve\, \phi\, d\mu_\ve
=
\nu^\ve\, \int \limits_{\mathbf{R}^d} \tilde{\chi}^\ve v^\ve\, \phi\, d\mu_\ve,
\ea
\ee
for any $\phi(z_1, \zeta) \in C_0^\infty(\mathbf{R}; C^\infty(\overline \Box))$.

We proceed in two steps.
First we choose an oscillating test function to determine the structure of $v^1(z_1,\zeta)$ in the two-scale limit in Lemma~\ref{lm:compactness-v^eps}.
Then we use a smooth test function of a slow argument and obtain the effective spectral problem.

Let us take
\begin{align*}
\Phi_\ve(z) & = \ve^{3/4}\, \varphi(z)\, \psi(\frac{z}{\ve^{3/4}}), \quad \varphi \in C_0^\infty(\mathbf{R}^d), \,\,\, \psi \in C^\infty(\Box),
\end{align*}
as a test function in \eqref{eq:weak-rescaled-prob-meas-bis}.
We consider all the terms separately.

%($\bullet$)

The gradient of $\Phi_\ve$ takes the form
\begin{align*}
\nabla \Phi_\ve(z) = \ve^{3/4}\, \psi(\frac{z}{\ve^{3/4}}) \nabla_z \varphi(z)  +
\varphi(z)\, \nabla_\zeta \psi(\zeta)\big|_{\zeta=z/\ve^{3/4}}.
\end{align*}
In the first term on the left hand side in \eqref{eq:weak-rescaled-prob-meas-bis}, with the help of the regularity properties of $a(z_1, \zeta)$, we can regard $\tilde a^\ve$ as a part of the test function.
Due to Lemma~\ref{lm:compactness-v^eps} we get
\begin{align*}
&\lim \limits_{\ve \to 0}\int \limits_{\mathbf{R}^d} \tilde{\chi}^\ve \tilde{a}^\ve \nabla v^\ve \cdot \nabla \Phi_\ve\, d\mu_\ve \\
& \quad = \frac{1}{|\Box|} \int \limits_{\mathbf{R}^d} \Big(\int \limits_{Y(0)} a(0,\zeta)\nabla_\zeta \psi(\zeta)\,d\zeta\Big) \cdot \nabla^{\mu_\ast} v(z_1, 0)\, \varphi(z_1, 0) d\mu_\ast \\
& \quad + \frac{1}{|\Box|} \int \limits_{\mathbf{R}^d} \Big(\int \limits_{Y(0)} a(0,\zeta)\nabla_\zeta \psi(\zeta)\cdot \nabla_\zeta v^1(z_1, \zeta)\,d\zeta\Big) \, \varphi(z_1, 0) d\mu_\ast.
\end{align*}
\item
Taking into account the regularity properties of $c, \varphi, \psi$ and the a priori estimate~\eqref{eq:apriori-est}, by Corollary~\ref{lm:MVT-1}, we get
\begin{align*}
& \Big|\frac{1}{\sqrt \ve}\, \ve^{3/4}\int \limits_{\mathbf{R}^d} \tilde{\chi}^\ve (\tilde{c}^\ve - \bar{c}(0))\, v^\ve\, \varphi(z)\, \psi(\frac{z}{\ve^{3/4}})\, d\mu_\ve \Big| \\
& \quad = \frac{\ve^{1/4}}{|\Box|}\Big|\int \limits_{\mathbf{R}^d} \Big(\int \limits_{Y(0)} (c(0, \zeta) - \bar{c}(0))\psi(\zeta)\,d\zeta\Big) v^\ve\, \varphi(z) d\mu_\ve  + O(\sqrt \ve)\Big|\le C\ve^{1/4},
\end{align*}
for some $C$ independent of $\ve$.

%($\bullet$)

Due to the boundedness of the eigenvalues of the rescaled problem (see estimate \eqref{eq:est-nu^eps}) and the normalization condition \eqref{eq:norm-cond-v^eps}, we have
\begin{align*}
\Big| \ve^{3/4} \nu^\ve\, \int \limits_{\mathbf{R}^d} \tilde{\chi}^\ve v^\ve\, \varphi(z)\, \psi(\frac{z}{\ve^{3/4}}) \, d\mu_\ve \Big| \le C\, \ve^{3/4}.
\end{align*}

Passing to the limit, as $\ve \to 0$, in \eqref{eq:weak-rescaled-prob-meas-bis} we obtain
\begin{align*}
& \int \limits_{\mathbf{R}^d} \Big(\int \limits_{Y(0)} a(0,\zeta)\nabla_\zeta v^1(z_1, \zeta)\cdot \nabla_\zeta\psi(\zeta)\, d\zeta\Big) \varphi(z_1, 0)\, d\mu_\ast \\
& \quad = - \int \limits_{\mathbf{R}^d} \Big(\int \limits_{Y(0)} a(0,\zeta)\nabla_\zeta \psi(\zeta)\,d\zeta\Big) \cdot \nabla^{\mu_\ast} v(z_1, 0)\, \varphi(z_1, 0) d\mu_\ast .
\end{align*}
Taking
\be
\label{eq:v^1}
v^1(z_1, \zeta)= N(\zeta) \cdot \nabla^{\mu_\ast} v(z_1, 0)
\ee
gives the following relation for the components of the vector-function $N(\zeta)$:
\begin{align*}
& \int \limits_{\mathbf{R}^d} \Big(\int \limits_{Y(0)} a(0,\zeta)\nabla N_k(\zeta)\cdot \nabla \psi(\zeta)\, d\zeta\Big) \varphi(z_1, 0)\, d \mu_\ast\\
& = - \int \limits_{\mathbf{R}^d} \Big(\int \limits_{Y(0)} a_{kj}(0,\zeta)\partial_{\zeta_j} \psi(\zeta)\, d\zeta\Big) \varphi(z_1, 0)\, d \mu_\ast,
\end{align*}
for any $\varphi \in C_0^\infty(\mathbf{R}^d)$, $\psi \in C^\infty(\Box)$. The last integral identity is a variational formulation associated to
\be
\label{eq:N-k}
\left\{
\begin{array}{lcr}
\displaystyle
- \mop{div}(a(0,y)\nabla N_k(y)) = \partial_{y_i} a_{i k}(0, y), \quad \hfill y \in Y(0),
\\[1.5mm]
\displaystyle
a(0, y)\nabla N_k(y)\cdot n = -  a_{i k}(0, y)\, n_i, \quad \hfill y \in \partial Y(0),\quad k = 1, 2, \cdots.
\end{array}
\right.
\ee
The existence and uniqueness of a periodic solution  $N_k(y) \in C^{1, \alpha}(\bar I; C^{1, \alpha}(\overline{Y(0)}))/\mathbf R$ to~\eqref{eq:N-k} follows from the Riesz representation theorem.

Using the representation \eqref{eq:v^1} together with the convergence (ii) in Lemma~\ref{lm:compactness-v^eps} we obtain
\begin{align*}
\tilde \chi^\ve \nabla v^\ve \overset{2}{\rightharpoonup} \chi(0, \zeta)\, (\nabla^{\mu_\ast} v(z_1, 0) + \nabla N(\zeta)\cdot \nabla^{\mu_\ast} v(z_1, 0)),\quad \ve \to 0.
\end{align*}
Now the structure of the function $v^1(z_1, \zeta)$ is known, and we can proceed by deriving the effective problem.
We use $\varphi(z) \in C_0^\infty(\mathbf R^d)$ as a test function in \eqref{eq:weak-rescaled-prob-meas-bis}.
Applying Corollary~\ref{lm:MVT-1} in the term containing $\tilde c^\ve - \bar c(0)$ we have
\begin{align*}
%\label{eq:aux-6}
& \int \limits_{\mathbf{R}^d} \tilde{\chi}^\ve(z) \tilde{a}^\ve(z) \nabla v^\ve(z) \cdot \nabla \varphi(z)\, d\mu_\ve \notag\\
& \qquad + \frac{1}{\sqrt{\ve}|\Box|} \int \limits_{\mathbf{R}^d} |Y(\ve^{1/4})| (\bar{c}(\ve^{1/4} z_1) - \bar{c}(0))\, v^\ve(z)\, \varphi(z)\, d\mu_\ve
+ O(\ve^{1/4}) \notag\\
& \quad = \nu^\ve\, \int \limits_{\mathbf{R}^d} \tilde{\chi}^\ve(z) v^\ve(z)\, \varphi(z)\, d\mu_\ve.
\end{align*}
The measure of $Y(\ve^{1/4}z_1)$, as a function of $z_1$, is a smooth function due to the properties of $F(x_1, y)$, defining the perforation.
Taylor expansions for $|Y(\ve^{1/4}z_1)|$ and $\bar{c}(\ve^{1/4} z_1) - \bar{c}(0)$ combined with the compactness result (i) in Lemma~\ref{lm:compactness-v^eps} give
\begin{align*}
\lim \limits_{\ve \to 0} \frac{1}{\sqrt{\ve}} \int \limits_{\mathbf{R}^d} \tilde \chi^\ve (\tilde{c}^\ve(z) - \bar{c}(0))\, v^\ve(z)\, \varphi(z)\, d\mu_\ve
& = \frac{|Y(0)|}{|\Box|}\, \frac{\bar c''(0)}{2}\, \int \limits_{\mathbf R^d}|z_1|^2 v(z_1, 0)\, \varphi(z_1, 0)\, d\mu_\ast.
\end{align*}
Here we use that $\varphi$ has a compact support, that is why the error term coming from the Taylor expansion vanishes.

Now we can pass to the limit in the integral identity \eqref{eq:weak-rescaled-prob-meas-bis} with $\varphi(z) \in C_0^\infty(\mathbf R^d)$ to obtain a problem for the pair $(v, \nabla^{\mu_\ast} v)$:
\begin{align*}
& \int \limits_{\mathbf R^d} \Big( \frac{1}{|Y(0)|} \int \limits_{Y(0)} a(0, \zeta)(I + \nabla N(\zeta))\, d\zeta\Big) \nabla^{\mu_\ast} v(z_1, 0)\cdot \nabla \varphi(z_1, 0)\,d\mu_\ast
 \notag\\
& \quad + \frac{\bar c''(0)}{2}\int \limits_{\mathbf R^d} |z_1|^2 v(z_1, 0)\varphi(z_1, 0)\, d\mu_\ast
 = \nu \int \limits_{\mathbf R^d} v(z_1, 0)\, \varphi(z_1, 0)\, d\mu_\ast.
\end{align*}
Here $\nabla N =\{\partial_{\zeta_i}N_j(\zeta)\}_{ij=1}^d$, and $I = \{\delta_{ij}\}_{ij=1}^d$ is the unit matrix.
Denote
\begin{align*}
A_{ij}^\eff = \frac{1}{|Y(0)|} \int \limits_{Y(0)} a_{ik}(0, \zeta)(\delta_{kj} + \partial_{\zeta_k}N_j(\zeta))\, d\zeta.
\end{align*}
In this way the limit problem in the weak form reads
\begin{align}
& \int \limits_{\mathbf R^d} A^\eff \nabla^{\mu_\ast} v(z_1, 0)\cdot \nabla \varphi(z_1, 0)\,d\mu_\ast
 + \frac{\bar c''(0)}{2}\int \limits_{\mathbf R^d} |z_1|^2 v(z_1, 0)\varphi(z_1, 0)\, d \mu_\ast \notag
 \\
& \quad = \nu \int \limits_{\mathbf R^d} v(z_1, 0)\, \varphi(z_1, 0)\, d\mu_\ast.
\label{eq:eff-weak-full}
\end{align}
As we know, the $\mu_\ast$-gradient is not unique, and one can see that the choice of a $A^\eff \nabla^{\mu_\ast} v(z_1, 0)$ is uniquely determined by the condition of orthogonality of the vector $A^\eff \nabla^{\mu_\ast} v$ to the subspace $\Gamma_{\mu_\ast}(0)$ of the gradients of zero. This can be shown by taking in \eqref{eq:eff-weak-full} any test function with zero trace $\varphi(z_1, 0, \cdots 0)=0$ and non-zero $\mu_\ast$-gradient, for example $\varphi(z)= \sum_{j\neq 1} z_j \psi_j(z_1)$ with arbitrary $\psi_j \in C_0^\infty(\mathbf R)\setminus \{0\}$. By the density of smooth functions, the subspace of vectors in the form $(0, \psi_2(z_1), \cdots, \psi_d(z_1))$, $\psi_j \in L^2(\mathbf R)$ is $\Gamma_{\mu_\ast}(0)$, and the condition of orthogonality to the gradients of zero gives that
$$
A^\eff \nabla^{\mu_\ast} v = (A_{1j}^\eff \partial_{z_j}^{\mu_\ast} v(z_1, 0), 0, \cdots, 0).
$$
If we define a solution of \eqref{eq:eff-weak-full} as a function $v(z) \in H^1(\mathbf R^d, \mu_\ast)$ satisfying the integral identity, then this solution is unique. A solution $(v, A^\eff \nabla^{\mu_\ast}v)$, as a pair,  is also unique due to the orthogonality to $\Gamma_{\mu_\ast}$. If one, however, defines a solution of \eqref{eq:eff-weak-full} as a pair $(v, \nabla^{\mu_\ast}v)$, then a solution is not unique. This has to do with the fact that the matrix $A^\eff$ is not positive definite, and the uniqueness of the flux does not imply the uniqueness of the gradient.

Next step is to prove that $A_{1j}^\eff = 0$ for all $j \neq 1$. To this end we rewrite the problem for $N_k$ in the following form:
\be
\label{eq:N-k-bis}
\left\{
\begin{array}{lcr}
\displaystyle
- \mop{div}(a(0,y)\nabla (N_k(y) + y_k) = 0, \quad \hfill y \in Y(0),
\\[1.5mm]
\displaystyle
a(0, y)\nabla (N_k(y) + y_k)\cdot n = 0, \quad \hfill y \in \partial Y(0),\quad k = 1, 2, \cdots.
\end{array}
\right.
\ee
Let us multiply the equation in \eqref{eq:N-k-bis} by $y_m$, $m \neq 1$, and integrate by parts over $Y(0)$. Note that for $m \neq 1$, the function $y_m$ is periodic in $y_1$ and can be used as a test function. This gives
\begin{align*}
\int \limits_{Y(0)} a(0, \zeta)\nabla (\zeta_k + N_k(\zeta)) \cdot \nabla \zeta_m \, d\zeta = 0,
\end{align*}
and since $\partial_{\zeta_j} \zeta_m = \delta_{jm}$, $A_{km}^\eff = 0$ for any $k=1, \cdots, d$ and $m \neq 1$.
Thus
\begin{align*}
A^\eff \nabla^{\mu_\ast} v = (A_{11}^\eff v'(z_1, 0), 0, \cdots, 0),
\end{align*}
and \eqref{eq:eff-weak-full} takes the form
\begin{align*}
& \int \limits_{\mathbf R} A_{11}^\eff v'(z_1, 0) \varphi'(z_1, 0)\,dz_1
 + \frac{\bar c''(0)}{2}\int \limits_{\mathbf R} |z_1|^2 v(z_1, 0)\varphi(z_1, 0)\, d z_1 \notag
 \\
& \quad = \nu \int \limits_{\mathbf R} v(z_1, 0)\, \varphi(z_1, 0)\, dz_1.
\end{align*}
Denoting $a^\eff = A_{11}^\eff$, $v(z_1)= v(z_1, 0)$, we see that the last integral identity is the weak formulation of \eqref{eq:eff-problem-bis}.

Using $N_i$ as a test function in \eqref{eq:N-k-bis} gives
\[
A_{ik}^\eff = \frac{1}{|Y(0)|} \int \limits_{Y(0)} a(0, \zeta)\nabla (\zeta_i +  N_i(\zeta))\cdot \nabla (\zeta_k + N_k(\zeta))\, d\zeta,
\]
which shows that $A^\eff$ is symmetric and positive semidefinite due to the corresponding properties of $a(x_1, y)$. If $e_1 = (1, 0,\cdots, 0)$,
\[
a^\eff = A_{11}^\eff = A^\eff e_1 \cdot e_1 \ge \frac{\Lambda_0}{|Y(0)|} \int \limits_{Y(0)} |\nabla (\zeta_1 +  N_1(\zeta))|^2\, d\zeta.
\]
Assuming that $\partial_{\zeta_i} (\zeta_1 + N_1(\zeta)) = 0$ for all $i$, leads to the contradiction since $N_1$ is periodic in $\zeta_1$. Thus, the effective coefficient $a^\eff$ is strictly positive. By standard arguments one can show that the spectrum of the limit problem \eqref{eq:eff-problem-bis} is real, discrete, all the eigenvalues are simple (see also Remark~\ref{rm:exact-sol}).

Due to the normalization condition \eqref{eq:norm-cond-v^eps} and the strong convergence in $L^2(\mathbf R^d, \mu_\ve)$,
\[
\lim \limits_{\ve \to 0} \frac{1}{\ve^{3(d-1)/4}} \int \limits_{\ve^{-1/4}G_\ve} \tilde \chi^\ve |v^\ve(z) - v(z_1)|^2\, dz = 0,
\]
the limit function $v(z_1)$ is not zero. Thus, $(\nu, v)$ is an eigenpair of the effective spectral problem \eqref{eq:eff-problem-bis}.

In this way, for a subsequence, any eigenvalue $\nu_j^\ve$ of \eqref{eq:rescaled-prob} converges to some eigenvalue $\nu_{J(j)}$ of the effective problem \eqref{eq:eff-problem-bis}, and the convergence of the corresponding eigenfunctions $v_j^\ve$ takes place. Lemma~\ref{lm:conv-v^eps} is proved.
\end{proof}

\begin{remark}
\label{rm:exact-sol}
The eigenpairs $(\nu_j, v_j)$ of the Sturm-Liouville problem
\begin{align*}
- a^\eff\, v'' + \frac{1}{2}\bar{c}''(0) \, z_1^2 \, v = \nu \, v, &&
v & \in L^2(\mathbf R),
\end{align*}
are
\begin{align*}
 \nu_j & = (2j - 1)\sqrt{\frac{a^\eff \bar c''(0)}{2}}, &
v_j & = H_j(\theta^{1/4} z_1) e^{- \sqrt{\theta} z_1^2/2 }, &
j & = 1,2, \ldots,
\end{align*}
where $\theta =  \frac{\bar c''(0)}{2a^\eff}$ and $H_j(x) = e^{x^2} \frac{{\rm d}^{(j-1)}}{{\rm d} x^{(j-1)}}e^{-x^2}$ the Hermite polynomials.
\end{remark}

%%%%%%%%%%%%%%%%%%%%%%%%%%%%%%%%%%%%%%%%%%%%%%%%%%%%%%%%%%%%%%%%%%%%%%%%%%%%%%%%%%%%%%%%%%%%%%%%%%%%%%%%%%%%%%%%%%%%%%%%%%%%%%%%%%%%%%%

\subsection{Convergence of spectra}
\label{sec:conv-spectra}
The goal of this section is to show that, for all $j$, the $j$th eigenvalue of the rescaled problem \eqref{eq:rescaled-prob} converges to the $j$th eigenvalue of the homogenized (effective) problem \eqref{eq:eff-problem-bis}, and the convergence of the corresponding eigenfunctions takes place.
The eigenvalues of the one-dimensional Sturm-Luiville problem~\eqref{eq:eff-problem-bis} are simple.
We will prove that, for sufficiently small $\ve$, the eigenvalues of \eqref{eq:rescaled-prob} are also simple.

\begin{lemma}
\label{lm:simplicity}
For sufficiently small $\ve$, along a subsequence, the eigenvalues of problem~\eqref{eq:rescaled-prob} are simple.
\end{lemma}
\begin{proof}
Suppose that some eigenvalue $\nu^\ve$ of \eqref{eq:rescaled-prob} has multiplicity two (or more), i.e. there exists two linearly independent eigenfunction $v_1^\ve, v_2^\ve$ corresponding to $\nu^\ve$.
Suppose also that $\nu^\ve$ converges, up to a subsequence, to $\nu_\ast$.
As was proved above, the eigenfunctions $v_1^\ve$, $v_2^\ve$ converge to the eigenfunctions $v_1, v_2$ corresponding to $\nu_\ast$. Since $\nu_\ast$ is simple, $v_1$ and $v_2$ should be linearly dependent, that is there exists $C\neq 0$ such that
\begin{align*}
v_1 + C\, v_2 = 0.
\end{align*}
Consider the linear combination $ v_1^\ve + C\, v_2^\ve$. It is an eigenfunction of \eqref{eq:rescaled-prob}.
Therefore,
\begin{align}
& \int \limits_{\mathbf{R}^d} \tilde{\chi}^\ve \tilde{a}^\ve \nabla (v_1^\ve + C\, v_2^\ve) \cdot \nabla (v_1^\ve + C\, v_2^\ve)\, d\mu_\ve + \frac{1}{\sqrt{\ve}} \int \limits_{\mathbf{R}^d} \tilde{\chi}^\ve (\tilde{c}^\ve - \bar{c}(0))\, (v_1^\ve + C\, v_2^\ve)^2\, d\mu_\ve \notag\\
& \quad = \nu^\ve\, \int \limits_{\mathbf{R}^d} \tilde{\chi}^\ve (v_1^\ve + C\, v_2^\ve)^2\, d\mu_\ve.\label{eq:aux-7}
\end{align}
Both eigenfunctions $v_1^\ve$ and $v_2^\ve$ are normalized by~\eqref{eq:norm-cond-v^eps}.
Thus,
\begin{align*}
\int \limits_{\mathbf{R}^d} \tilde{\chi}^\ve \tilde{a}^\ve \nabla v_i^\ve \cdot \nabla v_k^\ve\, d\mu_\ve +
\frac{1}{\sqrt{\ve}} \int \limits_{\mathbf{R}^d} \tilde{\chi}^\ve (\tilde{c}^\ve - \bar{c}(0))\, v_i^\ve \, v_k^\ve\, d\mu_\ve = \nu^\ve\, \delta_{ik}.
\end{align*}
The integral identity \eqref{eq:aux-7} takes the form
\be
\label{eq:aux-8}
2\, \nu^\ve = \nu^\ve\, \int \limits_{\mathbf{R}^d} \tilde{\chi}^\ve (v_1^\ve + C\, v_2^\ve)^2\, d\mu_\ve.
\ee
Due to the strong convergence in $L^2(\mathbf R^d, \mu_\ve)$,
\[
\|v_1^\ve + C\, v_2^\ve\|_{L^2(\mathbf R^d,\, \mu_\ve)} \to \|v_1 + C\, v_2\|_{L^2(\mathbf R^d,\, \mu_\ast)}= 0, \quad \ve \to 0.
\]
Passing to the limit on both sides of \eqref{eq:aux-8} yields
\[
2\, \nu_\ast = 0,
\]
which leads to a contradiction since $\nu_\ast \neq 0$.
Lemma~\ref{lm:simplicity} is proved.
\end{proof}

As a next step we prove that the order is preserved in the limit.
\begin{lemma}
\label{lm:conv-spectra}
For any $j$, the $j$th eigenvalue $\nu_j^\ve$ of problem \eqref{eq:rescaled-prob} converges to the $j$th eigenvalue $\nu_j$ of \eqref{eq:eff-problem-bis}, and the corresponding eigenfunctions converge in the sense of Lemma~\ref{lm:conv-v^eps}. In other words, $J(j)=j$ in Lemma~\ref{lm:conv-v^eps}.
\end{lemma}
\begin{proof}
By Lemma~\ref{lm:conv-v^eps}, all the eigenvalues of \eqref{eq:rescaled-prob} converge to some eigenvalues of \eqref{eq:eff-problem-bis}.
However, it is not proved yet that all the eigenvalues of the effective problem are limits of eigenvalues of \eqref{eq:rescaled-prob}.
We provide a proof by contradiction.
To fix the ideas, let us consider the first eigenvalue $\nu_1^\ve$ of \eqref{eq:rescaled-prob},
and assume that $\nu_1^\ve$ converges to the second eigenvalue $\nu_2$ of \eqref{eq:eff-problem-bis}.
The first eigenvalue $\nu_1^\ve$ is simple (in this case all other are simple too for small enough $\ve$, see Lemma~\ref{lm:simplicity})
and the corresponding eigenfunction $v_1^\ve$ converges to the eigenfunction $v_2$.

By the minmax principle,
\begin{align}
\label{eq:min-max-nu_1^eps}
\nu_1^\ve & = \inf \limits_{ \substack{v \in H^1(\mathbf R^d,\, \mu_\ve)\setminus\{0\}  \\ v|_{x_1 = \ve^{-1/4}\Gamma^\pm_\ve} = 0 }} \frac{\mathcal{F_\ve}(v)}{\|\tilde \chi^\ve v\|_{L^2(\mathbf R^d,\, \mu_\ve)}^{2}},\\
\mathcal{F_\ve}(v) & :=
\int\limits_{\mathbf R^d} \tilde \chi^\ve \tilde a^\ve \nabla v \cdot \nabla v \,d\mu_\ve
+ \frac{1}{\sqrt \ve} \int\limits_{\mathbf R^d} \tilde \chi^\ve (\tilde c^\ve - \bar c(0))\, v^2 \,d\mu_\ve,\notag
\end{align}
%\be
%\label{eq:min-max-nu_1^eps}
%\ba{l}
%\disp
%\nu_1^\ve =
%\inf \limits_{v \in H^1(\mathbf R^d,\, \mu_\ve)} \frac{\mathcal{F_\ve}}{\|\tilde \chi^\ve v\|_{L^2(\mathbf R^d,\, \mu_\ve)}^{2}},
%\\[2.5mm]
%\disp
%\mathcal{F_\ve} =
%\int\limits_{\mathbf R^d} \tilde \chi^\ve \tilde a^\ve \nabla v \cdot \nabla v \,d\mu_\ve
% + \frac{1}{\sqrt \ve} \int\limits_{\mathbf R^d} \tilde \chi^\ve (\tilde c^\ve - \bar c(0))\, v^2 \,d\mu_\ve,
% \ea
%\ee
where the infimum is taken over functions that vanish at the ends of the rescaled rod $\ve^{-1/4}\Gamma_\ve^\pm$.
The minimum is attained on the first eigenfunction $v_1^\ve$.
Since $\nu_1^\ve \to \nu_2$, as $\ve \to 0$, we can write $\nu_1^\ve = \nu_2 + o(1)$, as $\ve \to 0$.

We will construct a test function which gives a smaller value for the functional in \eqref{eq:min-max-nu_1^eps}. Let $v_1(z_1)$ be the first eigenfunction of \eqref{eq:eff-problem-bis} and $N$ the normalized solution of the auxiliary cell problem \eqref{eq:N-bis}.
Denote
\begin{align*}
V_\ve = \big(v_1(z_1) + \ve^{3/4} N\big(\frac{z}{\ve^{3/4}}\big)\, v_1'(z_1)\big)\, \phi_\ve(z_1),
\end{align*}
where $\phi_\ve(z_1)\in C^\infty(\mathbf R)$ is a cutoff which is equal to $1$ on $[-\frac{\ve^{-1/4}}{6}, \frac{\ve^{-1/4}}{6}]$, $\phi_\ve(z_1)=0$  on $\mathbf R \setminus [-\frac{\ve^{-1/4}}{3}, \frac{\ve^{-1/4}}{3}]$, and such that
\begin{align*}
0\le \phi_\ve\le 1, \quad |\phi_\ve'(z_1)| \le C\, \ve^{1/4}.
\end{align*}
We need to introduce this cutoff to make the test function $V_\ve$ satisfy the homogeneous Dirichlet boundary conditions on the ends of the rod.

We compute first the $L^2(\mathbf R^d, \mu_\ve)$-norm of $V_\ve$.
Taking into account the smoothness and exponential decay of $v_1(z_1)$, by Corollary~\ref{lm:MVT-1} we get
\begin{align*}
\|\tilde \chi^\ve V_\ve\|_{L^2(\mathbf R^d,\, \mu_\ve)}^2 =  \frac{|Y(0)|}{|\Box|} \int \limits_{\mathbf R^d} v_1(z_1)^2\, d\mu_\ve + O(\ve^{1/4}), \quad \ve \to 0.
\end{align*}
Then, substituting $V_\ve$ into the functional $\mathcal F_\ve$ in \eqref{eq:min-max-nu_1^eps}, up to the terms of higher order, we obtain
\begin{align*}
\mathcal F_\ve(V_\ve) & = \int \limits_{\mathbf R^d} \phi_\ve^2\, \tilde \chi^\ve\, \tilde a_{ij}^\ve \big( \delta_{1j} + \partial_{\zeta_j}N(\frac{z}{\ve^{3/4}})\big)\,
\big( \delta_{i1} + \partial_{\zeta_i}N(\frac{z}{\ve^{3/4}})\big) \, (v_1')^2\, d\mu_\ve
\\
& \quad +
\frac{1}{\sqrt \ve} \int \limits_{\mathbf R^d} \phi_\ve^2\, \tilde \chi^\ve (\tilde c^\ve - \bar c(0))\, v_1(z_1)^2\, d\mu_\ve
+ O(\ve^{1/4}).
\end{align*}
The function $v_1$, as an eigenfunction of the harmonic oscillator, decays exponentially (see Remark~\ref{rm:exact-sol}), thus the cutoff function does not contribute in the integral.
Moreover, by Corollary~\ref{lm:MVT-1}, using the properties of $\bar c$ and definition of the effective diffusion~\eqref{eq:a-eff},
\begin{align*}
\mathcal F_\ve(V_\ve) = \frac{|Y(0)|}{|\Box|}\Big(\int \limits_{\mathbf R} a^\eff \, (v_1'(z_1))^2\, d z_1
+
\frac{\bar c''(0)}{2} \int \limits_{\mathbf R} |z_1|^2\, v_1(z_1)^2\, dz_1 \Big)
+ O(\ve^{1/4}),
\end{align*}
as $\ve \to 0$.
Thus, by the minmax principle,
\begin{align*}
\frac{\mathcal{F_\ve}(V_\ve)}{\|\tilde \chi^\ve V_\ve\|_{L^2(\mathbf R^d,\, \mu_\ve)}^{2}}
& = \frac{\int_{\mathbf R} a^\eff \, (v_1'(z_1))^2\, d z_1 + \frac{\bar c''(0)}{2} \int_{\mathbf R} |z_1|^2\, v_1(z_1)^2\, dz_1 }{\|v_1\|_{L^2(\mathbf R)}^2} + O(\ve^{1/4}) \\
& = \nu_1 + O(\ve^{1/4}),
\end{align*}
as $\ve \to 0$.
We have shown that, on the one hand, the value of the infimum \eqref{eq:min-max-nu_1^eps} is close to $\nu_2$.
On the other hand, we have constructed a test function that gives a smaller value for the functional, $\nu_1 + O(\ve^{1/4})$.
Since $\nu_1^\ve$ is the smallest eigenvalue, we arrive at contradiction.

The argument can be repeated for any $j$. Lemma~\ref{lm:conv-spectra} is proved.
\end{proof}

We turn back to the original problem~\eqref{eq:orig-prob}.
By combining \eqref{eq:rescaling}--\eqref{eq:rescaling1} and using Lemmas~\ref{lm:conv-v^eps}, \ref{lm:simplicity}, \ref{lm:conv-spectra} we obtain the following theorem.
\begin{theorem}
\label{th:main-Th-full}
Suppose that (H1)--(H3) are satisfied.
Let $(\lambda_j^\ve, u_j^\ve)$ be the $j$th eigenpair to problem \eqref{eq:orig-prob}, and let $u_j^\ve$ be normalized by~\eqref{eq:norm-cond-u^eps}.
Then for any $j$, the following representation takes place:
\[
\lambda_j^\ve = \frac{\bar{c}(0)}{\ve} + \frac{\nu_j^\ve}{\sqrt{\ve}}, \quad u_j^\ve(x)= v_j^\ve\big(\frac{x}{\ve^{1/4}}\big),
\]
where $(\nu_j^\ve, v_j^\ve)$ defined by \eqref{eq:rescaled-prob} and normalized by \eqref{eq:norm-cond-v^eps} are such that
\begin{enumerate}[(i)]
\item For each $j$,
$\nu_j^\ve \to \nu_j$, as $\ve \to 0$.

\item
$v_j^\ve$ converges strongly in $L^2(\mathbf{R}^d, \mu_\ve)$ to $v_{j}(z_1)$ and
\[
\lim \limits_{\ve \to 0} \int \limits_{\mathbf R^d} \tilde \chi^\ve |v_j^\ve|^2\, d\mu_\ve
= \int \limits_{\mathbf R} |v_j|^2\, dz_1.
\]
Moreover, in $L^2(\rr^d, \mu_\ve)$,
\[
\tilde \chi^\ve \tilde a^\ve \nabla v_j^\ve \overset{2}{\rightharpoonup} \chi(0, \zeta)\, \Big((a^\eff v_j'(z_1), 0, \cdots, 0) + \nabla N(\zeta)\, \frac{{\mathrm d}v_{j}(z_1)}{{\mathrm d}z_1}\Big),\quad \ve \to 0,
\]
where $v_j(z_1)$ is the $j$th eigenfunction of the effective spectral problem \eqref{eq:eff-problem} corresponding to $\nu_j$.
\item[(iii)]
The fluxes converge weakly in $L^2(\mathbf R^d, \mu_\ve)$ (in the rescaled variables):
\[
a\big(\ve^{1/4}z_1, \frac{z}{\ve^{3/4}}\big) \nabla v_j^\ve \rightharpoonup (a^\eff v_j'(z_1), 0, \cdots, 0),\quad \ve \to 0.
\]
\item[(iv)]
For $\ve$ small enough, all the eigenvalues $\lambda_j^\ve$ are simple.
\end{enumerate}
\end{theorem}

%%%%%%%%%%%%%%%%%%%%%%%%%%%%%%%%%%%%%%%%%%%%%%%%%%%%%%%%%%%%%%%%%%%%%%%%%%%%%%%%%%%%%%%%%%%%%%%%%%%%%%%%%%%%%%%%%%%%%%%%%%%%%%%%%%%%%%%

\section{Integral estimate for rapidly oscillating functions}
\label{sec:MVT}

The purpose of this section is to give a proof of Corollary~\ref{lm:MVT-1}.

\begin{lemma}
\label{lm:MVT-0}
Let $v_\ve \in H^1_0(G_\ve, \Gamma_\ve)$ and $w(x_1, y) \in C^{1, \alpha}( \bar I; C^{\alpha}(\overline{\square}) )$.
Then as $\ve \to 0$,
\begin{align*}
\int \limits_{ \Omega_\ve } \!\! w\big(x_1, \frac{x}{\ve}\big)v_\ve^2(x)  \,dx
-
\frac{1}{|\Box|}\int \limits_{G_\ve } \! \int \limits_{Y(x_1)} \!\!\!\! w(x_1, y) \,dy \, v_\ve^2(x) \,dx
& = O(\ve \| v_\ve \|_{L^2(G_\ve )} \| \nabla v_\ve \|_{L^2(G_\ve )}).
\end{align*}
\end{lemma}

We use Lemma~\ref{lm:MVT-0} in the rescaled domain $\ve^{-1/4}G_\ve$ to estimate the zero-order term.
By performing a change of variables $z=x/\ve^{1/4}$, we obtain the following corollary.
\begin{corollary}
\label{lm:MVT-1}
Let $v_\ve \in H^1(\mathbf R^d, \mu_\ve)$ be such that $v_\ve = 0$ on $\ve^{-1/4}\Gamma_\ve^\pm$,
and $c(x_1, y) \in C^{1, \alpha}(\mathbf{R}; C^\alpha(\overline \Box))$.
Then as $\ve \to 0$,
\begin{align*}
& \int \limits_{ \mathbf R^d } (\chi \, c)\big(\ve^{1/4} z_1, \frac{z}{\ve^{3/4}}\big)\, v_\ve^2(z)  \,d\mu_\ve(z)
-
\frac{1}{|\Box|}\int \limits_{\mathbf R^d} |Y(\ve^{1/4} z_1)| \,\bar c(\ve^{1/4}z_1) \, v_\ve^2(z) \,d\mu_\ve(z) \\
&  \quad = O(\ve^{3/4} \| v_\ve \|_{L^2(\mathbf R^d,\, \mu_\ve)} \| \nabla v_\ve \|_{L^2(\mathbf R^d,\, \mu_\ve)}).
\end{align*}
\end{corollary}

\begin{proof}[Proof of Lemma~\ref{lm:MVT-0}]
To fix the argument in $Y(x_1)$ we
write the left hand side in the statement of Lemma~\ref{lm:MVT-0}
as a sum of integrals over the cells $\ve \square_j$, where $\square_j := (j, 0, \cdots, 0) + \square$.
\begin{align}
& \int\limits_{ \Omega_\ve } w\big(x_1, \frac{x}{\ve}\big)v_\ve^2(x)  \,dx
- \frac{1}{|\Box|}\int\limits_{G_\ve } \int\limits_{Y(x_1)} w(x_1, y) \,dy \, v_\ve^2(x) \,dx \notag \\
& \quad = \sum_j
\int\limits_{\ve \square_j}   ( \chi\big(x_1, \frac{x}{\ve}\big) - \chi\big(j, \frac{x}{\ve}\big) )w\big(x_1, \frac{x}{\ve}\big)v^2_\ve(x) \,dx \label{eq:firstint}\\
&  \qquad +
\frac{1}{|\Box|} \sum_j \int\limits_{\ve \square_j} \Big(  \int\limits_{Y(j)} w(x_1, y) \,dy - \int\limits_{Y(x_1)} w(x_1, y) \,dy  \Big)  v^2_\ve(x)  \,dx \label{eq:secondint}\\
& \qquad +
\int\limits_{G_\ve}  \sum_j \!\Big(  \chi\big(j, \frac{x}{\ve}\big)w\big(x_1,\frac{x}{\ve}\big) - \frac{1}{|\Box|}\int\limits_{Y(j)} w(x_1, y) \,dy  \Big)  v_\ve^2(x) \,dx. \label{eq:thirdint}
\end{align}
We estimate the terms on the lines \eqref{eq:firstint}--\eqref{eq:thirdint}.

We consider first \eqref{eq:firstint}.
Since $w$ is bounded we have by Lemmas~\ref{lm:annulus}, \ref{lm:lipschitzlayer},
\begin{align*}
& \Big| \int\limits_{\ve \square_j} ( \chi\big(x_1, \frac{x}{\ve}\big) - \chi\big(j, \frac{x}{\ve}\big) )w\big(x_1, \frac{x}{\ve}\big)v^2_\ve(x)   \,dx \Big| \\
& \quad \le C \ve^d \int\limits_{ A(j, C\ve) } v^2_\ve(\ve y) \,dy
\le C \Big( \ve \int\limits_{\ve \square_j} v_\ve^2 \,dx + \ve^2 \int\limits_{\ve \square_j} |v_\ve||\nabla v_\ve| \,dx \Big),
\end{align*}
for sufficiently small $\ve$, where $A(j,\eta) := \{ y \in \square_j : \mop{dist}(y,\Lambda(j)) < \eta \}$.
Thus as $\ve \to 0$,
\begin{align*}
& \sum_j
\int\limits_{\ve \square_j}   ( \chi\big(x_1, \frac{x}{\ve}\big) - \chi\big(j, \frac{x}{\ve}\big) )w\big(x_1, \frac{x}{\ve}\big)v^2_\ve(x)   \,dx \\
& \quad = O\Big(   \ve \int\limits_{ G_\ve } v_\ve^2 \,dx + \ve^2 \int\limits_{G_\ve} |v_\ve||\nabla v_\ve| \,dx  \Big)
= O\Big(  \ve \| v_\ve \|_{L^2(G_\ve)} \| \nabla v_\ve \|_{L^2(G_\ve)}  \Big),
\end{align*}
where we in the last step used the inequalities of arithmetic-geometric mean and Poincar\'e.

We turn to \eqref{eq:secondint}.
A Taylor expansion of $\int_{Y(s)} w(x_1, y)\,dy$ about $s = j$, which is justified since $\nabla_y F \neq 0$, yields
for $x \in \ve \square_j$ that
\begin{align*}
\int\limits_{Y(j)} w(x_1, y) \,dy - \int\limits_{Y(x_1)} w(x_1, y) \,dy
& = O( \ve ),
\end{align*}
as $\ve \to 0$, by the regularity of $F$ and boundedness of $w$.
A sum over $j$ gives the following estimate for \eqref{eq:secondint}:
\begin{align*}
& \sum_j \frac{1}{|Q|}\int\limits_{\ve \square_j}  \Big(  \int\limits_{Y(j)} w(x_1, y) \,dy - \int\limits_{Y(x_1)} w(x_1, y) \,dy  \Big)  v^2_\ve(x)  \,dx \\
& \quad = O\Big(  \ve \int\limits_{G_\ve} v_\ve^2 \,dx  \Big) = O( \ve \| v_\ve \|_{L^2(G_\ve)} \| \nabla v_\ve \|_{L^2(G_\ve)} ),
\end{align*}
as $\ve \to 0$, by the Poincar\'e inequality.

To estimate \eqref{eq:thirdint} we define a potential as follows.
Let $\Psi(x_1, y) \in C^{1, \alpha}(\bar I; C^{\alpha}(\overline{\square}))$
defined for $x_1 \in \bar I$ by
\[
\begin{cases}
\disp
{\mop{div}}_y \Psi(x_1, y) = \sum_j \Big( \chi(j, y) w(x_1, y) - \frac{1}{|\Box|}\int\limits_{\square_j} \chi(j, y) w(x_1, y) \,dy \Big), \,\, y \in \Box,\\
\Psi(x_1, y) \cdot n = 0, \text{  $y \in \partial \square$,}
\end{cases}
\]
where $n$ denotes the outward unit normal.
The compatibility condition is satisfied and
the right hand side belongs to $C^{1, \alpha}(\bar I; L^\infty(\square))$, so $\Psi$ is well-defined.

The integral \eqref{eq:thirdint} can then be estimated using the Green formula:
\begin{align*}
& \int\limits_{G_\ve}   \sum_j \Big(  \chi\big(j, \frac{x}{\ve}\big)w\big(x_1,\frac{x}{\ve}\big) - \frac{1}{|\Box|}\int\limits_{Y(j)} w(x_1, y) \,dy  \Big)  v_\ve^2(x) \,dx \\
& \quad = \int\limits_{G_\ve} {\mop{div}}_y \Psi(x_1, y)\big|_{y=x/\ve} \, v_\ve^2(x) \,dx \\
& \quad = \ve \int\limits_{G_\ve} \mop{div} \Psi\big(x_1, \frac{x}{\ve}\big) \, v_\ve^2(x) \,dx
- \ve \int\limits_{G_\ve} \partial_{x_1} \Psi(x_1, y)\big|_{y=x/\ve} \, v_\ve^2(x) \,dx \\
& \quad = - \ve \int\limits_{G_\ve} \Psi\big(x_1, \frac{x}{\ve}\big) \cdot \nabla v_\ve^2(x) \,dx
+ O\Big( \ve \int\limits_{G_\ve} v_\ve^2 \,dx \Big) \\
& \quad = O( \ve \| v_\ve \|_{L^2(G_\ve)} \| \nabla v_\ve \|_{L^2(G_\ve)} ),
\end{align*}
as $\ve \to 0$, by the Poincar\'e inequality.

We have estimated all terms to $O( \ve \| v_\ve \|_{L^2(G_\ve)} \| \nabla v_\ve \|_{L^2(G_\ve)}  )$ which gives the desired estimate.
\end{proof}

\begin{lemma}
\label{lm:annulus}
For any $s \in I$ there exists positive constants $C$, $\delta$ such that
\begin{align*}
Y(s) \Delta Y(t) & \subset \{ \, y \in \square : \mop{dist}(y, \Lambda(s)) < C|s - t| \, \},
\end{align*}
for all $t$ satisfying $|s - t| < \delta$. Here $\Delta$ denotes the symmetric difference for two sets, $\Lambda(s)=\{y \in \Box: \,\, F(s, y)=0\}$ is the boundary of the hole in $\Box$.
\end{lemma}
\begin{proof}
Let $s \in I$.
Let $y \in \Lambda(s)$ and denote by $n$ the outward unit normal to $Y(s)$ at $y$.
Let $\nu, \tau$ denote the normal-tangential components in local coordinates at $y$.

Since $F(s,y) = 0$ and $\partial_\nu F(s,y) = |\nabla_y F(s,y)| \neq 0$,
by the implicit function theorem, there exists  neighbourhoods $U$ of $(s,\tau(y))$ and $V$ of $\nu(y)$
and a unique function $G(x_1, \tau) : U \to V$ such that $F(x_1, G(x_1, \tau)n + \tau ) = 0$
for all $(x_1, \tau) \in U$.
Since $F(x_1, y) \in C^1(I \times \square)$ we have $G(x_1, \tau) \in C^1(U)$.

By Taylor's theorem $G(s, \tau) - G(t, \tau) = O(|s-t|)$ as $|s - t| \to 0$.
Thus there exists a neighbourhood $W$ of $y$ such that
\begin{align}
\label{eq:thinaux}
\sup_{w \in W} \mop{dist}(w, \Lambda(s)) < C |s - t|,
\end{align}
for $|s - t|$ sufficiently small.

Since $y$ was arbitrary there is an open cover $\{W_j\}$ of $\Lambda(s)$.
From the regularity of $F$ it follows that $\Lambda(s)$ is compact.
Thus there exists a finite subcover and in particular a finite set of positive constants $\{ C_j \}$
as in \eqref{eq:thinaux}.
Therefore,
\begin{align*}
Y(s) \Delta Y(t) \subset \{ y \in \square : \mop{dist(y, \Lambda(s))} < \max_j C_j |s - t| \},
\end{align*}
for $t$ sufficiently close to $s$. Since $s$ was arbitrary we are done.
\end{proof}

The following version of the trace inequality will be used.

\begin{lemma}\label{lm:lipschitzlayer}
Let $\Omega$, $\Omega'$ be bounded Lipschitz domains in $\rr^d$ such that $\overline{\Omega'} \subset \Omega$.
Denote $A(\Omega', \delta) = \{ y : \mop{dist}(y, \partial \Omega') < \delta \}$.
Then there exists an positive constant $C_{\Omega'}$ such that
\begin{align*}
\int\limits_{A(\Omega', \delta)} v^2 \, dx
& \le
C_{\Omega'}
\delta
\Big(
\int\limits_{\Omega} v^2 \, dx
+
\int\limits_{\Omega} |v||\nabla v| \, dx
\Big),
\end{align*}
for all sufficiently small $\delta > 0$ and all $v \in H^1(\Omega)$.
\end{lemma}
\begin{proof}
Since $\Omega'$ is Lipschitz there exists positive constants $C_1, C_2$,
a finite number of uniformly Lipschitz functions $a^j : [-C_1,C_1]^{d-1} \to \rr$ and
rotated-translated coordinates $y^j = (\tilde y^{j}, y_d^j)$ such that all $x \in \partial \Omega'$ are represented as $(\tilde y^j, a(\tilde y^j))$ for some $j$, and
\begin{align*}
\{ y : a^j(\tilde y^{j}) < y_d^j < a^j(\tilde y^j) + C_2 \} & \subset \Omega', \\
\{ y : a^j(\tilde y^{j}) - C_2 < y_d^j < a^j(\tilde y^j) \} & \subset \Omega \setminus \overline{\Omega'}.
\end{align*}
Since $\overline{\Omega'} \subset \Omega$, the cells
$V^j = \{ y : |y^j_k| \le C_1, \, 1 \le k \le d, \, |y^j_d| \le C_2 \}$
can be chosen to be contained in $\Omega$.
We drop the superscript $j$ and consider a local cell $V$.

Let $\delta$ be small enough such that $0<\delta < C_2$
and let $v \in C^1(\overline{\Omega})$.
Then for $\tilde y \in [-C_1,C_1]^{d-1}$ and
$a(\tilde y) - \delta < y_d < \tau < a(\tilde y) + \delta$ we have
\begin{align*}
& (y_d - a(\tilde y) + C_2) v^2(y) \\
& \quad =
- \int\limits_{y_d}^\tau  \partial_{y_d}( (y_d - a(\tilde y) + C_2 )v^2(y) )  \, dy_d
+ (\tau - a(\tilde y) + C_2)v^2(\tilde y, \tau) \\
& \quad = - \int\limits_{y_d}^\tau ( v^2(y) + 2(y_d - a(\tilde y) + C_2)v(y) \partial_{y_d} v (y) )  \, dy_d
+ (\tau - a(\tilde y) + C_2)v^2(\tilde y, \tau) \\
& \quad \le
\int\limits_{a(\tilde y) - C_2}^{a(\tilde y) + C_2} ( v^2(y) + 3C_2  |v(y)| |\nabla v (y)| )  \, dy_d
+ 3C_2 v^2(\tilde y, \tau).
\end{align*}
An integration in $\tau$ over $(a(\tilde y)-C_2, a(\tilde y)+ C_2)$ gives
\begin{align*}
(y_d - a(\tilde y) + 2\delta) v^2(y)
& \le
3
\int\limits_{a(\tilde y) - C_2}^{a(\tilde y) + C_2} v^2 \,dy_d
+
4C_2
\int\limits_{a(\tilde y) - C_2}^{a(\tilde y) + C_2} |v||\nabla v| \,dy_d.
\end{align*}
Dividing by $y_d - a(\tilde y) + 2\delta$ and integrating in $y$ over $(-C_1,C_1)^{d-1} \times (-\delta,\delta)$ yields
\begin{align*}
\int\limits_{V \cap A(\Omega', \delta)} v^2 \,dx
& \le
\ln \frac{C_2 + \delta}{C_2 - \delta}
\Big(
3
\int\limits_{V} v^2 \, dx
+
4C_2 \int\limits_{V} |v||\nabla v| \,dx
\Big).
\end{align*}
Using $\ln  \frac{C_2 + \delta}{C_2 - \delta} = O(\delta)$ as $\delta \to 0$,
a sum over all cells $V^j$ gives a $C_{\Omega'} > 0$ such that
\begin{align*}
\int\limits_{A(\Omega', \delta)} v^2 \,dx
& \le
C_{\Omega}' \delta
\Big(
\int\limits_{\Omega} v^2 \, dx
+
\int\limits_{\Omega} |v||\nabla v| \,dx
\Big).
\end{align*}
The result follows by the density of $C^1(\overline{\Omega})$ in $H^1(\Omega)$.
\end{proof}

%%%%%%%%%%%%%%%%%%%%%%%%%%%%%%%%%%%%%%%%%%%%%%%%%%%%%%%%%%%%%%%%%%%%%%%%%%%%%%%%%%%%%%%%%%%%%%%%%%%%%%%%%%%%%%%%%%%%%%%%%%%%%%%%%%%%%%%%%%%%%%%%

\section{One-dimensional example}\label{exa:oned}
\label{sec:1D-example}

To illustrate the asymptotics of the eigenpairs to equation \eqref{eq:orig-prob},
we consider, for small $\ve > 0$, the principal eigenpair $(\lambda_1^\ve, u_1^\ve)$ to the following equation
 on the one-dimensional interval (0,1):
\begin{align}\label{eq:onedu}
\begin{cases}
-u'' + \frac{1+x^2}{\ve}u = \lambda u, & (|x| < 1),\\
u(-1) = u(1) = 0.
\end{cases}
\end{align}
 A change of variable $x = \ve^{1/4}y$ leads to an equation with no big coefficients:
\begin{align}\label{eq:onedv}
\begin{cases}
-v'' + y^2 v = \mu^\ve v, & \big(|y| < \ve^{-1/4}\big),\\
v(-\ve^{-1/4}) = v(\ve^{-1/4}) = 0,
\end{cases}
\end{align}
where $v(y) = u(\ve^{1/4}y)$ and $\mu^\ve = \sqrt{\ve}\big(\lambda - \frac{1}{\ve}\big)$.
To describe $u_1^\ve$ as $\ve \to 0$ we will use the equation
\begin{align}\label{eq:onedw}
-w'' + y^2 w = \eta w,\,\,\,\,\,\, \big(y \in \rr).
\end{align}
The principal eigenpair to \eqref{eq:onedw} is
$\eta_1 = 1$, $w_1 = \exp(-y^2/2)$.
Under suitable normalization of $u_1^\ve$, $w_1$,
\begin{align}\label{eq:onedconv}
\left|\lambda_1^\ve - \frac{1}{\ve} - \frac{\eta_1}{\sqrt{\ve}} \right| & \to 0, &
\sup_{x \in \rr} \left|u_1^\ve(x) - w_1\left(\frac{x}{\ve^{1/4}}\right)\right| &\to 0.
\end{align}
as $\ve \to 0$, where $u_1^\ve$ is extended to $\rr$ by zero.
In particular, this means that $u_1^\ve$ concentrates in the vicinity of zero, which is the unique minimum point of the coefficient $1 + x^2$ in \eqref{eq:onedu}.

We prove the convergences in \eqref{eq:onedconv}.
Let $(\mu^\ve_1, v_1^\ve)$, $(\eta_1, w_1)$ be the principal eigenpairs to the equations
\eqref{eq:onedv}--\eqref{eq:onedw}, respectively, normalized by \mbox{$v_1^\ve(0) = w_1(0) = 1$}.
The eigenfunctions $v_1^\ve$, $w_1$ are simple and do not change sign and are thus by the symmetry of the boundary value problems necessarily even.
Therefore they are uniquely defined by the following initial value problems on $[0,\ve^{-1/4}]$
for fixed $\mu_1^\ve$, $\eta_1$:
\begin{align}
\begin{cases}
(v_1^\ve)'' = (y^2 - \mu^\ve_1)v_1^\ve,\\
v_1^\ve(0) = 1,\\
(v_1^\ve)'(0) = 0,
\end{cases}
&&
\begin{cases}
w_1'' = (y^2 - \eta_1)w_1,\\
w_1(0) = 1,\\
w_1'(0) = 0.
\end{cases}
\label{eq:onedvw}
\end{align}

A comparison of \eqref{eq:onedv} and \eqref{eq:onedw} using the minimum principle for the Rayleigh quotient gives $\mu^\ve_1 > \eta_1$ since $w_1 = \exp(-y^2/2) > 0$. In particular,
 the coefficients in the equations in \eqref{eq:onedvw} satisfy
$y^2 - \mu_1^\ve < y^2 - \eta_1$, and both are increasing functions.
Thus by \eqref{eq:onedvw}, the graph of $v_1^\ve$ will stay below the graph of $w_1$ and the distance between the graphs is an increasing function.
We conclude that
\begin{align*}
\sup_y |w_1 - v_1^\ve| & = w_1(\ve^{-1/4}) = e^{-\frac{1}{2\sqrt{\ve}}},
\end{align*}
which shows the second convergence in \eqref{eq:onedconv}, and in addition that the error decreases exponentially with $\ve$, as $\ve \to 0$.
We turn to the eigenvalues by considering the Rayleigh quotient:
\begin{align*}
\mu_1^\ve & = \inf_{ \substack{v \in H^1_0(-\ve^{-1/4},\,\ve^{-1/4} )\\ v \neq 0} }
\frac{ \int\limits_{-\ve^{-1/4}}^{\ve^{-1/4}} (v')^2 \,dy + \int\limits_{-\ve^{-1/4}}^{\ve^{-1/4}} y^2 v^2 \,dy }{\int\limits_{-\ve^{-1/4}}^{\ve^{-1/4}} v^2 \,dy}.
\end{align*}
Using $v = |1 - \ve^{1/4}y|w_1(y)$ as a test function,
where $1 - \ve^{1/4}y$ is the polynomial of degree one that makes the test function satisfy the boundary condition, gives
\begin{align*}
\mu_1^\ve & \le
\frac{ 2(4 + 3\sqrt{\ve})e^{\frac{1}{\sqrt{\ve}}} \int\limits_0^{\ve^{-1/4}} \! e^{-s^2} ds + 4\ve^{\frac{1}{4}}e^{\frac{1}{\sqrt{\ve}}} + 2\ve^{\frac{1}{4}}  }{
2(4 + \sqrt{\ve})e^{\frac{1}{\sqrt{\ve}}} \int\limits_0^{\ve^{-1/4}} \! e^{-s^2} ds + 4\ve^{\frac{1}{4}}e^{\frac{1}{\sqrt{\ve}}} + 2\ve^{\frac{1}{4}}
 }
\to 1,
\end{align*}
as $\ve \to 0$.
Together with $\eta_1 < \mu_1^\ve$, this shows the first convergence in~\eqref{eq:onedconv}.

The graphs of $u_1^\ve(x)$ and $w_1(\ve^{-1/4}x)$ for some values of $\ve$ are shown in Figure~\ref{fig:oned}.
We see that the error decreases with $\ve$ and the functions concentrate.
\begin{figure}[h]
\centering
\includegraphics[width=.99\textwidth]{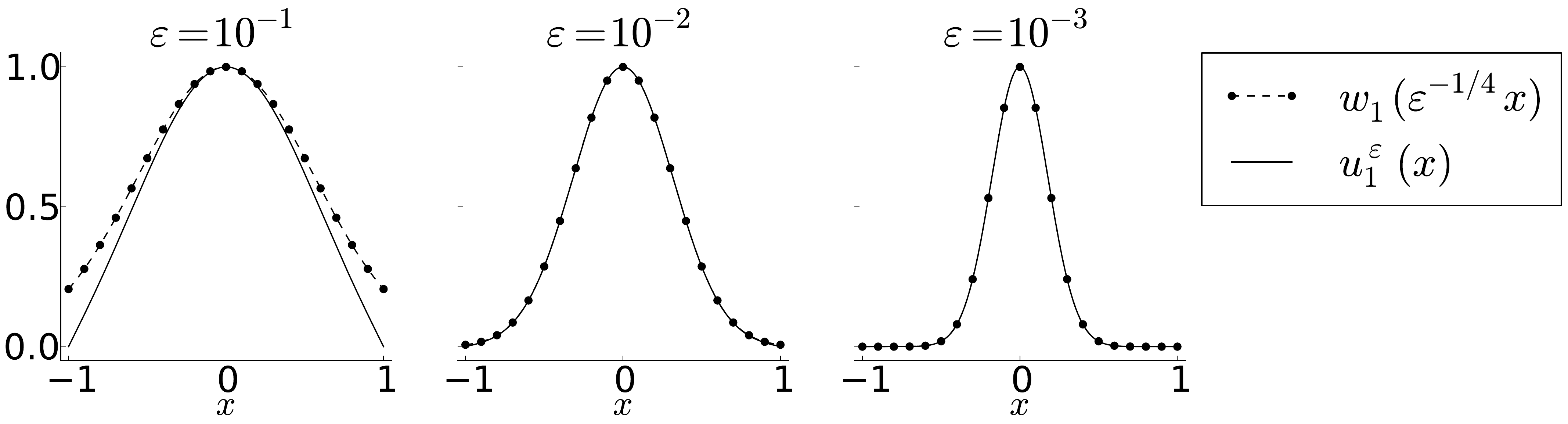}
\caption{Graphs of eigenfunctions and approximations in one-dimensional example.}
\label{fig:oned}
\end{figure}

%\section*{Acknowledgements}
%We thank professor Andrey Piatnitski for his valuable comments and discussions on the paper.

\FloatBarrier

\def\cprime{$'$} \def\cprime{$'$} \def\cprime{$'$}

\end{document}